\documentclass[12pt]{amsart}
\usepackage[margin=30mm]{geometry}
\usepackage[all]{xy}
\usepackage{amssymb}
\usepackage{amsmath}
\usepackage{bigints}
\usepackage{bbm}
\usepackage{graphics,graphicx}
\usepackage{ esint }
\usepackage{stfloats}
\usepackage{eepic}
\usepackage{epsfig}
\usepackage{breqn}
\usepackage{stmaryrd}
\usepackage{hyperref}
\usepackage{wrapfig}
\usepackage{enumerate}
\usepackage[usenames]{color}
\usepackage[toc,page]{appendix}

\usepackage{enumitem}

\newcommand{\R}{\mathbb{R}}

\newcommand{\N}{\mathbb{N}}
\renewcommand{\H}{\mathcal{H}}
\renewcommand{\S}{\mathcal{S}}

\newcommand{\C}{\mathcal{C}}
\newcommand{\K}{\mathcal{K}}
\newcommand{\p}{\mathcal{P}}

\newcommand{\Om}{\Omega}
\newcommand{\om}{\omega}
\newcommand{\eps}{\varepsilon}

\newtheorem{theorem}{Theorem}

\newtheorem{proposition}[theorem]{Proposition}
\newtheorem{corollary}[theorem]{Corollary}
\newtheorem{definition}[theorem]{Definition}
\newtheorem{remark}[theorem]{Remark}

\usepackage{caption} 
\usepackage{tikz}
\newtheorem{conjecture}{Conjecture}

\renewcommand{\H}{\mathcal{H}}
\renewcommand{\S}{\mathbb{S}}

\newcommand{\J}{\mathcal{J}}

\newcounter{mnotecount}[section]

\newcommand{\rmnote}[1]{}

\title{Optimal $L^p$-approximation of convex sets by convex subsets}

\author{
  Zakaria Fattah 
  \and
  Ilias Ftouhi 
  \and
  Enrique Zuazua
}

\begin{document}
\date{\today}

\maketitle

{\centering\footnotesize \textit{Dedicated to the memory of Professor Hammadi Bouslous}.\par}

\begin{abstract}
Given a convex set $\Omega$ of $\mathbb{R}^n$, we consider the shape optimization problem of finding a convex subset $\omega\subset \Omega$, of a given measure, minimizing the $p$-distance functional
$$\J_p(\omega) :=  \left(\int_{\S^{n-1}} |h_\Omega-h_\omega|^p d\H^{n-1}\right)^{\frac{1}{p}},$$
where $1 \le p <\infty$ and $h_\omega$ and $h_\Omega$ are the support functions of $\omega$ and the fixed container $\Omega$, respectively.

We prove the existence of solutions and show that this minimization problem $\Gamma$-converges, when $p$ tends to $+\infty$, towards the problem of finding a convex subset $\omega\subset \Omega$, of a given measure, minimizing the Hausdorff distance to the convex $\Omega$.

In the planar case, we show that the free parts of the boundary of the optimal shapes, i.e., those that are in the interior of $\Om$, are given by polygonal lines.

Still in the $2-d$ setting, from a computational perspective, the classical method based on optimizing Fourier coefficients of support functions is not efficient, as it is unable to efficiently capture the presence of segments on the boundary of optimal shapes. We subsequently propose a method combining Fourier analysis and a numerical scheme recently introduced in \cite{beni_num},  allowing to obtain  accurate results, as demonstrated through numerical experiments.

\end{abstract}

\section{Introduction and main results}

The strategic placement and shape design of sensors and actuators is of paramount importance in numerous applications involving  Partial Differential Equation (PDEs) models or purely geometric ones as they influence system's behavior and enhance its observability and controllability properties. From a mathematical viewpoint, a multitude of interesting problems can be formulated in the context of optimal design, aiming to identify the subdomains minimizing a certain energy functional measuring sensor/actuator performance, which could be purely geometric or related to the PDEs under consideration, while satisfying some natural constraints related, for instance, to their shape, number of connected components, size, etc.

This topic has been widely explored across diverse contexts and by various scientific communities, highlighting its broad applicability and theoretical significance. Here, we provide a non-exhaustive list of typical related problems and references:


\begin{itemize}
      \item The optimal shape design of actuators within a given set under the framework of optimal control is well-documented. Notable references include \cite{PrivatTrelatZuazua13,zbMATH06440180,zbMATH06722979}, which discuss various approaches and methodologies for actuator placement that optimize control effectiveness.
  \item The minimization of the average distance within a subset, first introduced in 2002, is a classic problem \cite{zbMATH01996419,zbMATH05058778}. A comprehensive overview of this problem and its developments can be found in \cite{zbMATH06110561}, which also presents a detailed state of the art. More recent studies have shifted focus towards minimizing the maximal distance, expanding the theoretical and practical implications of this problem \cite{zbMATH,zbMATH06996635,Ftouhi_zuazua, zbMATH06919706}.
\item The strategic placement of cavities within a region to either minimize or maximize specific eigenvalues of a differential operator has attracted significant attention. This area of study offers rich literature and an array of methodologies. For an up-to-date overview, readers may refer to \cite{ftouhi_steklov}, which includes a comprehensive introduction along with extensive bibliographical references. 
\item It is also worth noting that such point of view has also been considered for problems of mathematical biology such as the optimal distribution of resources so as to maximize population size \cite{zbMATH07155098,zbMATH07168708} and the mathematical analysis of the optimal habitat configurations for species persistence \cite{zbMATH05240724}.  
    

\end{itemize}

The diverse applications and theoretical depth of these problems underline the importance of developing robust mathematical methods for actuators optimal design and placement. This article aims to contribute to this ongoing research by presenting new findings that enhance the understanding of optimal design strategies in systems governed by purely geometric criteria. It
naturally extends the study presented in \cite{Ftouhi_zuazua}, where the authors studied the problem of optimal shape design of a convex subset $\omega$ of a given measure contained in a convex container $\Omega\subset\mathbb{R}^n$ in such a way to minimize the Hausdorff distance $d^H(\om,\Om)$ between the two sets.  

This problem is mathematically formulated as follows 
\begin{equation}\label{prob:hausdorff_zuazua}
    \min\{d^H(\om,\Om)\ |\ \omega\subset \Om\ \text{is convex and}\ |\om|=c \},
\end{equation}
with $c\in[0,|\Om|]$ and $d^H(\om,\Om)$ the Hausdorff distance between the sets $\omega$ and $\Omega$ defined as follows
\begin{equation}\label{eq:def_hausdorff}
d^H(\om,\Om) := \max\left\{ \sup_{x\in\om} \inf_{y\in \Om} \|x-y\|, \sup_{x\in\Om} \inf_{y\in \om} \|x-y\|\right\}.    
\end{equation}

Manipulating formula \eqref{eq:def_hausdorff} presents clear challenges as it involves the infinity norm of the distance function which is non differentiable. To address this, the authors of \cite{Ftouhi_zuazua} consider an analytic approach consisting of parameterizing convex sets through their support functions. 

The support function of a convex set $\Omega\subset\mathbb{R}^n$ can be defined as a function of $\R^n$ as follows: $$h_\Om:x\in \R^n\to \sup\{\langle x,y \rangle\ |\ y\in \Omega\}.$$
By the  definition, one can see that the support function $h_\Om$ is convex on $\R^n$ and thus also continuous.

Since the function $h_\Om$ satisfies the scaling property $h_\Om(t x) = t h_\Om(x)$ for $t>0$, it is completely determined by its values on the unit sphere $\mathbb{S}^{n-1}$. Therefore, from now on, we will consider support functions defined on $\mathbb{S}^{n-1}$ as follows: 
\begin{equation}\label{def:support_intro}
    h_\Om:\theta\in \mathbb{S}^{n-1} \longrightarrow  \sup\{\langle \theta,y \rangle\ |\ y\in \Omega\}.
\end{equation}

For a brief presentation of support functions and their key properties, we refer the reader to Section \ref{ss:support}. 

The perspective of parameterizing sets via their support functions enables the reformulation of the geometric problem \eqref{prob:hausdorff_zuazua} into an analytical one as the Hausdorff distance between two convex sets $\om$ and $\Om$ is given by 
\begin{equation}\label{eq:norm_infinie}
  \J_\infty(\omega) :=  d^H(\om,\Om)= \|h_\Om-h_\om\|_\infty:=\max_{\theta\in\mathbb{S}^{n-1}}|h_\Om(\theta)-h_\om(\theta)|,
\end{equation}
where $h_\omega,h_\Omega: \mathbb{S}^{n-1}\longrightarrow \mathbb{R}$ correspond to the support functions of $\omega$ and $\Omega$ respectively, defined on the unit sphere $\mathbb{S}^{n-1}$. This approach is also used to propose and implement a scheme for the numerical resolution of problem \eqref{prob:hausdorff_zuazua}, see \cite[Section 5]{Ftouhi_zuazua}.

Nevertheless, being non differentiable, the infinity norm in \eqref{eq:norm_infinie} raises  numerical and theoretical challenges. To avoid such differentiability issues, it is then natural to consider smooth approximations of the Hausdorff distance
via the following $p$-distance between convex sets
\begin{equation}\label{p_distance}
    \J_p(\omega) :=  \|h_\Omega-h_\omega\|_p := \left(\int_{\mathbb{S}^{n-1}} |h_\Omega-h_\omega|^p d\mathcal{H}^{n-1}\right)^{\frac{1}{p}},
\end{equation}
with $p\ge 1$ and $\mathcal{H}^{n-1}$ being the $(n-1)$-dimensional Hausdorff measure on the sphere $\mathbb{S}^{n-1}$. \vspace{2mm}

The functionals $(\mathcal{J}_p)$ have the advantage of being shape differentiable, with explicit formulas for the shape derivatives, which is essential for numerical optimization. \vspace{1mm} 

In the present work, we propose to introduce and study the following approximated problems 
$$(\p_p):\ \ \ \ \ \sigma_p :=\inf\{\J_p(\omega)\ |\ \omega\subset \Om\ \text{is convex and}\ |\om|=c \},$$
with $c\in[0,|\Om|]$. 

Before presenting the main contributions of the present paper, it is worth noting that the $p$-distances introduced in \eqref{p_distance} provide classical metrics on the space of convex bodies. Such metrics have been investigated by R. A. Vitale in \cite{zbMATH03958124} for $p\ge 1$ and by A. Florian in \cite{zbMATH04117342} for $p=1$. We also refer to a relatively more recent paper \cite{henrot_harrel} by A. Henrot and E. Harrel who studied shape optimization problems involving the $p$-distance functionals for $p=2$ and $p=\infty$. 

On an other note, one might naturally expect the solutions of the approximated problems $(\p_p)$ to be smooth, and easier to approximate numerically, using the classical scheme based on optimizing the Fourier coefficients of support functions such as in \cite[Section 5]{Ftouhi_zuazua}.  Nevertheless, as we shall see in the present paper, the situation is trickier as in the planar setting the solutions are proven to present some singularities since their boundaries are shown to contain polygonal parts. 

The first main result of the paper is concerned with proving the $\Gamma-$convergence of the problems $(\p_p)$ to the problem of minimizing the Hausdorff distance \eqref{prob:hausdorff_zuazua}. This result is proved for any dimension. 
\begin{theorem}\label{th:main_gamma}
Let $\Omega\subset\mathbb{R}^n$ be a convex body, $p\in[1,+\infty)$ and $c\in[0,|\Omega|]$. 
\begin{itemize}
    \item The problem $(\p_p)$ admits solutions. \vspace{1mm} 
    \item The sequence of functionals $(\J_p)$ $\Gamma$-converges to $\J_\infty$ when $p$ tends to $+\infty$. Therefore,
    $$\lim\limits_{p\rightarrow +\infty} \sigma_p = \sigma_\infty :=\inf\{d^H(\omega,\Om)\ |\ \omega\subset \Om\ \text{is convex and}\ |\om|=c \}$$
    and every accumulation point, with respect to the Hausdorff distance $d^H$,  of solutions of Problems $(\p_p)$ solves problem $$    \min\{d^H(\om,\Om)\ |\ \omega\subset \Om\ \text{is convex and}\ |\om|=c \}.$$
\end{itemize}

\end{theorem}

In the planar case, we are able to prove the following structural result on the boundary of the optimal shapes: 
\begin{theorem}\label{th:main_gamma_2}
Let $\Omega\subset\mathbb{R}^2$ be a planar convex body, $p\in[1,+\infty)$ and $c\in[0,|\Omega|]$. 
If $\om^*$ is a solution of the problem $(\p_p)$, then the free part of its boundary, i.e., $\partial \om^*\backslash \partial \Om$, is the union of polygonal lines. In particular, if the container $\Omega$ is a polygon then $\omega^*$ is also a polygon. 
\end{theorem}

Theorem \ref{th:main_gamma_2} is restricted to the planar case because its proof relies on results of \cite{lnp}, where the authors provide sufficient conditions on the cost functionals  guaranteeing the presence of polygonal parts in the boundary of optimal sets. Such results are, up to our knowledge, still not available in higher dimensions due to non-trivial technical challenges as explained in \cite{zbMATH06548314}.

The proof relies on combining the arguments of \cite{lnp} with the following important technical result that we formulate in a general setting for possible use in other related problems. 

\begin{theorem}\label{th:main_equivalence}
  Let $J$ and $F$ be two shape functionals and $\mathcal{C}$ a given class of subsets of $\R^n$ endowed with a distance $\delta$. Set $I:=\{ F(\Om)\ |\ \Om\in \C \}$ and assume that: 
  \begin{enumerate}[label=(\Alph*)]
      \item\label{hyp11} The class $\mathcal{C}$ is non-empty and compact with respect to the distance $\delta$.  
      \item\label{hyp22} The functionals $F$ and $J$ are not constant and are continuous on $\C$ with respect to the distance $\delta$.   
      \item\label{hyp33} For every $\Om\in\C$, there exists $\eps_\Om>0$ and a continuous map $$\Psi_\Om:x\in [\inf I,\sup I]\cap (F(\Om)-\eps_\Om,F(\Om)+\eps_\Om)\longmapsto \Psi_\Om(x)\in \C$$ such that
      $$\Psi_\Om(F(\Om))=\Om \ \ \text{and}\ \ \forall x\in [\inf I,\sup I]\cap (F(\Om)-\eps_\Om,F(\Om)+\eps_\Om),\ \ F(\Psi_\Om(x)) = x.$$
      \item\label{hyp44} A set $\Om\in \C$ such that $F(\Om)<\sup I$ cannot be a local minimizer of $J$ in $(\mathcal{C},\delta)$.
  \end{enumerate}

  Under these assumptions, we have the following properties: 

  \begin{itemize}
  \item The set $I$ is a closed interval with non-empty interior: it is exactly given by the closure of the interval $(\inf I, \sup I)$. \vspace{1mm}
  \item For every $x\in I$, the problem 
  $$\inf\{J(\Om)\ |\ \Om\in \mathcal{C}\ \text{and}\ F(\Om)=x\}$$ admits solutions.\vspace{1mm}
  \item The function  
  \begin{equation}\label{eq:f}
    f:x\in I\longrightarrow \min\{J(\Om)\ |\ \Om\in \mathcal{C}\ \text{and}\ F(\Om)=x\}  
  \end{equation}
  is continuous and strictly decreasing.\vspace{1mm} 
  \item The following problems are equivalent: 
 \begin{itemize}
    \item  $\min\{J(\Om)\ |\ \Om\in \C\ \text{and}\ F(\Om) = x\}$,
    \item $\min\{J(\Om)\ |\ \Om\in \C\ \text{and}\ F(\Om) \leq x\}$,
    \item $\min\{F(\Om)\ |\ \Om\in \C\ \text{and}\ J(\Om) = f(x)\}$,
    \item $\min\{F(\Om)\ |\ \Om\in \C\ \text{and}\ J(\Om) \leq f(x)\}$.
\end{itemize}   

 \end{itemize}   
\end{theorem}

Let us now comment on this result:
\begin{itemize}
    \item Theorem \ref{th:main_equivalence} provides sufficient conditions  implying the continuity and the monotonicity of functions defined as infima (or maxima), such as the function $f$ defined in \eqref{eq:f}. Such a result leads to the equivalence between different optimization problems,  allowing, in various cases, to simplify the initial one, by considering an equivalent  more convenient formulation. Such ideas have been considered in different frameworks such as in the study of Blaschke--Santal\'o diagrams, see for example \cite[Corollary 3.13]{ftouh}, the problem of the minimization of the Hausdorff distance treated in \cite{Ftouhi_zuazua} or the study of the equivalence of minimal time and minimal norm controls in \cite{zuazua_equivalence}.
    \item Hypothesis \ref{hyp33} of Theorem \ref{th:main_equivalence} corresponds to a perturbation property than can be expressed as follows: given a set $\Omega$ in the class $\C$, it is in general possible to continuously perturb it while remaining in $\C$, so as to increase or decrease its corresponding energy $F(\Omega)$.
    \item In the same spirit, Hypothesis \ref{hyp44} is also a perturbation property (weaker than the latter one) for the second functional $J$. This assumption is crucial for the monotonicity of the function $f$ defined in \eqref{eq:f}. 
\end{itemize}

\textbf{Outline of the paper:} The paper is organized as follows: in Section \ref{s:notations}, we introduce the notations and recall the basic properties of support functions of convex sets. The proofs of Theorems \ref{th:main_gamma}, \ref{th:main_gamma_2} and \ref{th:main_equivalence} are given in Section \ref{s:proofs}. Then, some results on the extremal cases $p=1$ and $p=\infty$ are presented in Section \ref{s:extremal}. Finally, Section \ref{s:numerics} focuses on the numerical resolution of the problems, implementing and comparing two distinct parametrization methods across various test cases.





\section{Notations and useful results}\label{s:notations}

\subsection{Notations}\label{ss:notations}
\begin{itemize}
    \item If $X$ and $Y$ are two subsets of $\R^n$, the Hausdorff distance between $X$ and $Y$ is defined as follows
    $$d^H(X,Y) = \max(\sup_{x\in X}d(x,Y),\sup_{y\in Y}d(y,X)),$$
    where $d(a,B):= \inf\limits_{b\in B}\|a-b\|$ quantifies the distance from the point $a$ to the set $B$. Note that when $\om\subset \Om$, as it is the case in the problems considered in the present paper, the Hausdorff distance is given by 
    $$d^H(\om,\Om) := \sup_{x\in\Om} d(x,\om).$$
     \item If $\Om$ is a convex set, then $h_\Om$ corresponds to its support function.     \item For $p\in[1,+\infty)$, we take $\J_p(\omega) :=  \|h_\Omega-h_\omega\|_p=\left(\int_{\mathbb{S}^{n-1}} |h_\Omega-h_\omega|^p d\mathcal{H}^{n-1}\right)^{\frac{1}{p}}$. \vspace{1mm} \item $\J_\infty(\omega) := d^H(\omega,\Omega) = \|h_\Omega-h_\omega\|_\infty = \max_{\theta\in\mathbb{S}^{n-1}}|h_\Om(\theta)-h_\om(\theta)|$.

    \item Given a convex set $\Om$ and $t\ge 0$, we denote by $\Om_{-t}$ its inner parallel set at the distance $t$, which is defined by 
    $$\Om_{-t}:=\{x\ |\ d(x,\partial \Om)\ge t\}.$$
    \item Given a convex set $\Om$ and $c\in[0,|\Omega|]$, we denote by $\K_c$ the class of convex bodies of measure $c$ included in $\Omega$.
    \item The Minkowski sum of two convex sets $\Om_1$ and $\Om_2$ is given by 
    $$\Om_1+ \Om_2 := \{x+y\ |\ x\in\Om_1\ \text{and}\ y\in\Om_2\}.$$
    \item $H^1_{\text{per}}(0,2\pi)$ is the set of $H^1$ functions that are $2\pi$-periodic.
\end{itemize}

\subsection{Properties of support functions in the planar case}\label{ss:support}
According to the definition of the support function given in \eqref{def:support_intro}, in the planar case, the support function can be defined as follows: 
\begin{definition}\label{def:support}
The support function of a planar bounded convex set $\Om$ is defined on $[0,2\pi)$ as follows: 
$$h_\Om:[0,2\pi)\longrightarrow \sup\left\{\left\langle \binom{\cos{\theta}}{\sin{\theta}}, y \right\rangle\ |\ y\in\Om\right\}.$$
\end{definition}

\begin{figure}[!h]
    \centering
    \includegraphics[scale=.8]{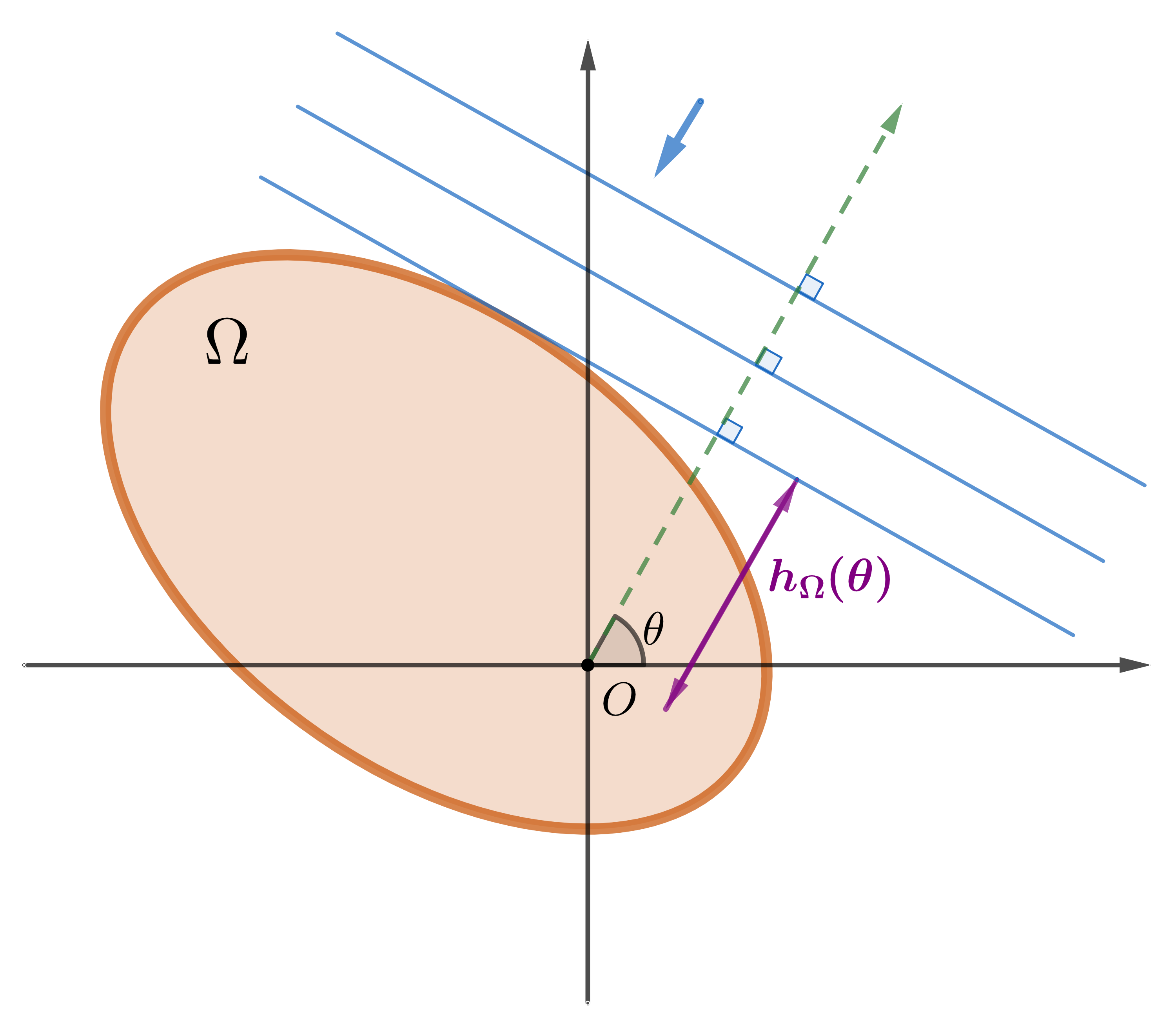}
    \caption{The support function of the convex $\Omega$.}
    \label{fig:support_function}
\end{figure}

The support functions of planar convex sets have some interesting properties: 
    \begin{itemize}
        \item They allow to provide a simple criterion of the convexity of $\Om$. Indeed, $\Om$ is convex, if and only if, $h_\Om''+h_\Om\ge 0 $ in the sense of distributions, see for example \cite[Formula (2.60)]{schneider}.  
        \item They behave linearly for the Minkowski sum and dilatation. Indeed, if $ \Om_1$ and $\Om_2$ are two convex bodies and $\alpha, \beta >0$, we have  
    $$h_{\alpha \Om_1 + \beta \Om_2}= \alpha h_{\Om_1}+ \beta h_{\Om_2},$$
    see \cite[Section 1.7.1]{schneider}.
        \item They allow to parametrize the inclusion in a simple way. Indeed, if $\Om_1$ and $\Om_2$ are two convex sets, we have 
        $$\Om_1\subset \Om_2\Longleftrightarrow h_{\Om_1}\leq h_{\Om_2}.$$
        \item They also provide elegant formulas for some relevant geometric quantities. For example, the perimeter and the area of a planar convex body $\Om$ are respectively given by 
        $$P(\Om) = \int_0^{2\pi}h_\Om(\theta)d\theta,  \ \ |\Om|= \frac{1}{2}\int_0^{2\pi} h_\Om (h_\Om''+h_\Om) d\theta=\frac{1}{2} \int_0^{2\pi} ({h_\Om'}^2-h_\Om^2)d\theta$$
        and the Hausdorff distance between two convex bodies $\Om_1$ and $\Om_2$ is given by 
    $$d^H(\Om_1,\Om_2)=\max_{\theta\in [0,2\pi)} |h_{\Om_1}(\theta)-h_{\Om_2}(\theta)|,$$
    see for example \cite[Lemma 1.8.14]{schneider}. 
    \end{itemize}

\section{Proofs of Theorems \ref{th:main_gamma}, \ref{th:main_gamma_2} and \ref{th:main_equivalence}}
\label{s:proofs}

\subsection{Proof of Theorem \ref{th:main_gamma}}

In this section we present the proof of Theorem \ref{th:main_gamma} and a technical equivalence result in Proposition \ref{prop:equivalence_p}, which will play a crucial role in the proof of Theorem \ref{th:main_gamma_2}.  

For convenience of the reader, before presenting the proof, let us recall the statement of Theorem \ref{th:main_gamma}:
\begin{theorem}\label{prop:gamma}
Let $\Om\subset \R^n$ be a convex body. For every $p\in[1,+\infty)\cup \{+\infty\}$ and $c\in[0,|\Om|]$, the problem
$$(\p_p):\ \ \ \ \ \sigma_p := \inf\{\J_p(\omega)\ |\ \omega\subset \Om\ \text{is convex and}\ |\om|=c \}$$ 
admits solutions. 

Moreover, the sequence of functionals $(\J_p)$ $\Gamma$-converges to $\J_\infty$ as $p$ tends to $+\infty$. Therefore
\begin{itemize}
    \item $\lim\limits_{p\rightarrow +\infty} \sigma_p = \sigma_\infty.$ 
    \item Every accumulation point, with respect to the Hausdorff distance, of solutions to problems $(\p_p)$ is a solution of $(\p_\infty)$. 
\end{itemize} 
\end{theorem}
\begin{proof}
The existence of optimal shapes is quite standard and is obtained by following the fundamental method of calculus of variation. Let us then focus on the proof of the $\Gamma$-convergence. 

Let $(\om_p)$ be a sequence of elements of the class $\K_c$, whichconverges to a certain $\om$ with respect to the Hausdorff distance. We have 
\begin{equation}\label{eq:convergence_gamma}
    \lim\limits_{p\rightarrow +\infty} \J_p(\omega_p) = \J_\infty(\omega).
\end{equation}
Indeed, 
\begin{align*}
|\J_p(\omega_p)-\J_\infty(\omega)| &\leq   |\J_p(\omega_p)-\J_p(\omega)| + |\J_p(\omega)-\J_\infty(\omega)|  \\
&= \big| \|h_\Om-h_{\om_p}\|_p - \|h_\Om-h_{\om}\|_p\big| + \big| \|h_\Om-h_{\om}\|_p - \|h_\Om-h_{\om}\|_\infty\big| \\
&\leq \|h_{\om_p}-h_\om\|_p + \big| \|h_\Om-h_{\om}\|_p - \|h_\Om-h_{\om}\|_\infty\big| \\
&\leq \|h_{\om_p}-h_\om\|_\infty + \big| \|h_\Om-h_{\om}\|_p - \|h_\Om-h_{\om}\|_\infty\big|\\
&= d^H(\om_p,\om) + \big| \|h_\Om-h_{\om}\|_p - \|h_\Om-h_{\om}\|_\infty\big|\\
&\underset{p\rightarrow +\infty}{\longrightarrow} 0. 
\end{align*}
We are now in position to check the classic $\Gamma$-convergence conditions:
\begin{itemize}
    \item For every sequence $(\om_p)$ of elements of $\K_c$, which converges to $\om$ with respect to the Hausdorff distance, by \eqref{eq:convergence_gamma} we have
    $$\J_\infty(\om) = \lim\limits_{p\longrightarrow +\infty} \J_p(\om_p) =\liminf\limits_{p\longrightarrow +\infty} \J_p(\om_p)$$ 
    \item On the other hand, for every $\om\in\K_c$, we consider the (constant) sequence $(\omega_p):=(\om)$. It trivially converges to $\om$. We have 
    $$\J_\infty(\om) = \lim\limits_{p\longrightarrow +\infty} \J_p(\om) =\lim\limits_{p\longrightarrow +\infty} \J_p(\om_p)=\limsup\limits_{p\longrightarrow +\infty} \J_p(\om_p).$$ 
\end{itemize}
Therefore, the sequence of functionals $(\J_p)$ $\Gamma$-converges to $\J_\infty$ when $p$ tends to $+\infty$. Moreover, this sequence is equi-coercive on the metric space $(\K_c,d^H)$, which is compact with respect to the Hausdorff distance, and for all $m\in \R$ the class $\K_c$, contains the sublevel set $\{\om\in\K_c\ |\ \J_p(\om)\leq m\}$. We then conclude by using the fundamental theorem of $\Gamma$-convergence, see for example \cite[Theorem 1.3.1]{gamma}. 
\end{proof}

Let us now state the following property that provides the equivalence between Problem $(\p_p)$ and another one that falls within the framework of \cite{lnp}. This allows in Theorem \ref{th:main_gamma_2} to prove the existence of polygonal parts in the boundary of optimal sets. 

\begin{proposition}\label{prop:equivalence_p}
  For every $p\in[1,+\infty)\cup\{+\infty\}$ and $c\in[0,|\Omega|]$, Problem  $(\p_p)$ is equivalent to the problem
$$\min\{\ |\om|\ |\ \text{$\om$ is convex, included in $\Om$ and}\  \J_p(\om)=d_c\},$$
where $d_c$ is a constant depending on $c$. 
\end{proposition}
\begin{proof}
  We use Theorem \ref{th:main_equivalence}, with:
  \begin{itemize}
      \item $\C$ being the class of closed convex sets included in $\Om$,
      \item $\delta$ being the Hausdorff distance denoted by $d^H$,
      \item $F$ being the area functional $|\cdot|$  and $J$ being the functional $\J_p$.
  \end{itemize}

  Let us check the hypotheses of Theorem \ref{th:main_equivalence}: 
  \begin{enumerate}
      \item Hypothesis \ref{hyp11} is a direct consequence of the boundedness of $\Om$ and the Blaschke selection theorem, see for example \cite[Theorem 1.8.7]{schneider}.  \vspace{1mm}
      \item It is classical that the area functional $|\cdot|$ is continuous on the class of closed convex sets endowed with the Hausdorff distance. As for $\J_p$, one can easily check its continuity with respect to the Hausdorff distance by writing 
      \begin{align*}
|\J_p(\om_1)-\J_p(\om_1)| &=  \big|\|h_\Om-h_{\om_1}\|_p-\|h_\Om-h_{\om_1}\|_p\big|  \\
&\leq\|h_{\om_1}-h_{\om_2}\|_p \leq  \|h_{\om_1}-h_{\om_2}\|_\infty= d^H(\om_1,\om_2).
\end{align*}
      Moreover, both functionals are not constant as $\Om$ has non-empty interior. Thus, Hypothesis \ref{hyp22} is satisfied. \vspace{1mm}
      \item    Let $\om\subset \Om$ and $|\om| = x_0$. We consider the following continuous map 
    $$ \Psi_\om: x\longmapsto \Psi_\om(x) = 
\left\{\begin{matrix}
\om_{-\tau_x}\ \ \ \ \ \ \ \ \ \ \ \ \ \ \ \ \ \ \ \ \ \ \ \ \  \ \ \text{if $x\in [0,x_0]$},\vspace{2mm}
\\ 
(1-t_x)\om_{x_0}+t_x\Om
\ \ \ \ \ \ \ \ \text{if $x\in (x_0,|\Om|]$},
\end{matrix}\right.  $$
where $\tau_x$ is chosen in $\R^+$ in such a way that
$$|\om_{-\tau_x}| = x$$
and $t_x$ is chosen in $[0,1]$ such that
$$|(1-t_x)\om+t_x\Om|=x.$$
 The map $\Psi_\om$ satisfies Hypothesis \ref{hyp33}. \vspace{2mm}

 \item Let $\om\subset\Om$ such that $|\om|\in (0,|\Om|)$. The set $\om$ is then different than the container $\Om$. This yields $\J_p(\om)>0$.\\
 For $\eps>0$ sufficiently small, we consider $\om_\eps:= (1-\eps)\om + \eps \Om$, which is strictly included in $\Om$. We have 
\begin{align*}
\J_p(\om_\eps) &=  \|h_\Om - h_{\om_\eps}\|_p  = \|h_\Om - h_{(1-\eps)\om + \eps \Om}\|_p \\
&= \|h_\Om - (1-\eps)h_\om -\eps h_\Om\|_p=(1-\eps)\|h_\Om-h_\om\|_p\\ &< \|h_\Om-h_\om\|_p= \J_p(\om).
\end{align*}
 This means that $\om$ is not a local minimizer of $\J_p$ in the class $\C$ which shows that Hypothesis \ref{hyp44} is satisfied.  
  \end{enumerate}

\end{proof}

\subsection{Proof of Theorem \ref{th:main_gamma_2}}
The following result is restricted to the planar case as its proof relies on results of \cite{lnp} which, to our knowledge, are still not available in higher dimensions due to non-trivial technical challenges as explained in \cite{zbMATH06548314}.

\begin{theorem}\label{prop:polygons}
 Let $\Omega$ be a planar convex body, $p\in[1,+\infty)$ and $c\in[0,|\Om|]$. If $\om^*$ is a solution of the problem $(\p_p)$, then $\partial \om^*\cap \Om$ is the union of polygonal lines.    
\end{theorem}
 \begin{proof}
By Proposition \ref{prop:equivalence_p}, Problem  $(\p_p)$ is equivalent to the problem
$$\min\{|\om|\ |\ \text{$\om$ is convex, included in $\Om$ and}\  \J_p(\om)=f(c)\}$$
that can be reformulated in terms of support functions as follows
\begin{equation}\label{prob:ilias}
    \min\left\{\frac{1}{2}\int_0^{2\pi}(h^2-{h'}^2)d\theta\ \Big|\ h''+h\ge 0,\ \ h\leq h_\Om\ \ \text{and}\ \ \int_0^{2\pi} (h_\Om-h)^{p} d\theta = f(c) \right\}.
\end{equation}

Following the notations of \cite{lnp}, we set
\begin{itemize}
    \item $\mathbb{T}:=[0,2\pi)$.
    \item For every $h$ such that $h''+h\ge 0$ and $h\leq h_\Om$, 
$$\mathbb{T}_{in}(h):=\{\theta\in \mathbb{T}\ |\ h(\theta) < h_\Om(\theta)\}.$$\
    \item $m(h):= \int_0^{2\pi}(h_\Omega-h)^p d\theta.$ 
    \item $j(h):= \frac{1}{2}\int_0^{2\pi} (h^2-{h'}^2)d\theta$.
\end{itemize}

We have
$$m''(h)(v,v) = \frac{p(p-1)}{2}\int_0^{2\pi} (h_\Omega-h)^{p-2} v^2d\theta. $$
Thus 
$$\|m''(h)(v,v)\|\leq \beta \|v\|^2_{L^2(0.2\pi)} \leq \beta \|v\|^2_{H^s(0,2\pi)}.$$

On the other hand, the second order derivative of the volume is given by 
$$j''(h)(v,v) = \frac{1}{4} \left(\int_0^{2\pi} v^2 d\theta - \int_0^{2\pi} {v'}^2 d\theta\right).$$

We note that the functional $j$ satisfies property (49) of \cite[Proposition 4.10]{lnp}. Namely
$$j''(h)(v,v) \leq -\frac{1}{4}|v|^2_{H^1(0,2\pi)} + \frac{1}{4}\|v\|^2_{H^s(0,2\pi)}.$$

Therefore, by \cite[Theorem 2.9]{lnp}, if $I$ is a connected component of $\mathbb{T}_{in}(h)$, then it is a sum of finite Dirac masses. Finally, we conclude that for any optimal shape for Problem \eqref{prob:ilias} (and thus also the initial problem $(\p_p)$), the  free parts of its boundary (i.e., that are included in the interior of the container $\Omega$) are given by polygonal lines.  
\end{proof}\vspace{4mm}

\begin{remark}
    Enlightened by the result of Theorem \ref{th:main_gamma_2}, it is not difficult to see that if $\Om$ is a polygon, then the optimal solution $\om^*$ is also a polygon as the part $\partial \om^*\cap \partial \Om$ will simply be the union of a finite number of lines. We note that this property is similar to the one proved in \cite[Proposition 8]{Ftouhi_zuazua} with a different and more elementary method for the $(\p_\infty)$ problem. 
\end{remark}

\subsection{Proof of Theorem \ref{th:main_equivalence}}

Let $J$ and $F$ be two shape functionals and $\mathcal{C}$ a given class of subsets of $\R^n$ endowed with a distance $\delta$. We consider $I:=\{ F(\Om)\ |\ \Om\in \C \}$. We assume that the following hypotheses, stated in Theorem \ref{th:main_equivalence}, hold:
  \begin{enumerate}[label=(\Alph*)]
      \item\label{hyp1} The class $\mathcal{C}$ is non-empty and compact with respect to the distance $\delta$.  
      \item\label{hyp2} The functionals $F$ and $J$ are not constant and are continuous on $\C$ with respect to the distance $\delta$.   
      \item\label{hyp3} For every $\Om\in\C$, there exists $\eps_\Om>0$ and a continuous map $$\Psi_\Om:x\in [\inf I,\sup I]\cap (F(\Om)-\eps_\Om,F(\Om)+\eps_\Om)\longmapsto \Psi_\Om(x)\in \C$$ such that
      $$\Psi_\Om(F(\Om))=\Om \ \ \text{and}\ \ \forall x\in [\inf I,\sup I]\cap (F(\Om)-\eps_\Om,F(\Om)+\eps_\Om),\ \ \ \  F(\Psi_\Om(x)) = x.$$
      \item\label{hyp4} A set $\Om\in \C$ such that $F(\Om)<\sup I$ cannot be a local minimizer of $J$ in $(\mathcal{C},\delta)$.
  \end{enumerate}

For clarity, the proof of Theorem \ref{th:main_equivalence} is structured as a sequence of propositions, each addressing a distinct assertion of the theorem.
\begin{proposition}
  The set $I=\{F(\Om)\ |\ \Om\in \mathcal{C}\}$ is a closed interval with non-empty interior.   
\end{proposition}
\begin{proof}
  Since the class $\C$ is non-empty (by Hypothesis \ref{hyp1}) and $F$ is not a constant functional (by Hypothesis \ref{hyp2}), the set $I$ is non-empty and $\inf I < \sup I$.
  \begin{itemize}

      \item First, let us show that the set $I$ is closed. Let $(x_n)$ be a sequence of elements of $I$ converging to some $x^*$. We consider a sequence $(\Om_n)$ of elements of $\C$ such that 
      $$\forall n\in \mathbb{N}, \ \ \ F(\Om_n) = x_n.$$
      Since the class $(\C,\delta)$ is compact (Hypothesis \ref{hyp1}), then there exists $\Om^*\in \C$ such that $(\Om_n)$ converges to $\Om^*$ with respect to $\delta$ up to a subsequence still denoted by $(\Om_n)$. \\By the continuity of $F$ with respect to $\delta$, we deduce that 
      $$F(\Om^*) = \lim\limits_{n\rightarrow +\infty} F(\Om_n) = \lim\limits_{n\rightarrow +\infty} x_n = x^*,$$ which shows that $x^*\in I$. This proves that the set $I$ is closed. 
      \item  It remains to show that the set $I$ is an interval. To do so, we prove that $(\inf I,\sup I)\subset I$. Let $y\in (\inf I,\sup I)$ and consider $$a_{y}:= \sup\{x\in I,\ x\leq y\}.$$ Since the set $I$ is closed, we have $a_{y}\in I$, i.e., there exists $\Om_0\in \C$ such that $F(\Om_0) = a_{y}$. \\
      If we assume that $a_{y} < y$, then by Hypothesis \ref{hyp3}, there exists $\eps>0$ and a continuous map $$\Psi:x\in [a_{y},a_{y}+\eps)\to \Psi(x)\in \C$$ such that
      $$\forall x\in [a_{y},a_{y}+\eps),\ \ \ \ \ \ F(\Psi(x)) = x.$$ 
      This is in contradiction with $a_{y}$ being the supremum of the set $\{x\in I,\ x\leq y\}$. Thus, $y = a_{y}\in I$. We then conclude that the set $I$ is an interval.
  \end{itemize}
\end{proof}

\begin{proposition}\label{prop:existence}
Let $x\in I$. The problem $$\min\{J(\Om)\ |\ \Om\in \mathcal{C}\ \text{and}\ F(\Om)=x\}$$ admits solutions.
\end{proposition}

\begin{proof}

Let $(\Om_n)$ be a minimizing sequence for the problem under consideration, i.e., such that $\Om_n\in \C$, $F(\Om_n) = x$ and  $$\lim_{n\rightarrow +\infty} J(\Omega_n) =\inf\{J(\Om)\ |\ \Om\in \mathcal{C}\ \text{and}\ F(\Om)=x\}.$$

By Hypothesis \ref{hyp1}, $\C$ is compact with respect to $\delta$. Thus, the sequence $(\Om_n)$ converges up to a subsequence (that we also denote by $(\Om_n)$) to a set $\Om^*\in \C$ as $n$ tends to $+\infty$.

By the continuity of $F$ and $J$ with respect to the distance $\delta$ (Hypothesis \ref{hyp2}), we have
$$F(\Om^*) = \lim_{n\rightarrow +\infty} F(\Om_n) = x$$
and
$$J(\Om^*) = \lim_{n\rightarrow +\infty} J(\Om_n) = \inf\{J(\Om)\ |\ \Om\in \mathcal{C}\ \text{and}\ F(\Om)=x\}.$$
We then conclude that the set $\Om^*$ solves the problem $$\inf\{J(\Om)\ |\ \Om\in \mathcal{C}\ \text{and}\ F(\Om)=x\}.$$
\end{proof}

\begin{proposition}\label{prop:continuity}
The function  $f:x\in I\to \min\{J(\Om)\ |\ \Om\in \mathcal{C}\ \text{and}\ F(\Om)=x\}$ is continuous and strictly decreasing on $I$.       
\end{proposition}
\begin{proof}
    \textbf{\textit{Continuity:}}

Let $x_0\in I$. By Proposition \ref{prop:existence}, for every $x\in I$, there exists $\Om_x$ solution of the problem 
$$\min\{J(\Om)\ |\ \Om\in\mathcal{C}\ \text{and}\ F(\Om)=x\}.$$
\begin{itemize}
    \item We first show an  inferior limit inequality. Let $(x_n)_{n\ge 1}$ be a sequence converging to $x_0$ such that 
    $$\liminf_{x\rightarrow x_0} J(\Om_x)=\lim_{n\rightarrow +\infty} J(\Om_{x_n}).$$ 
    Since all the convex sets $\Om_{x_n}$ are included in the compact class $\C$ and the functionals $J$ and $F$ are continuous with respect to the distance $\delta$ (Hypothesis \ref{hyp2}), there exists $\Om^*\in\C$ that is a limit of a subsequence still denoted by $(\Om_{x_n})$ such that $F(\Om^*) = x$.
    We then have
    $$f(x_0)\leq J(\Om^*) = \lim_{n\rightarrow+\infty} J(\Om_{x_n}) = \liminf_{x\rightarrow x_0} J(\Om_x)=\liminf_{x\rightarrow x_0} f(x).$$
    \item It remains to prove a superior limit inequality.  Let $(x_n)_{n\ge 1}$ be a sequence converging to $x_0$ such that 
    $$\limsup_{x\rightarrow x_0} f(x)=\lim_{n\rightarrow +\infty} f(x_n).$$
    By Hypothesis \ref{hyp3}, there exists $\eps_0>0$ and a continuous map $$\Psi_0:x\in I\cap (x_0-\eps_0,x_0+\eps_\Om)\longmapsto \Psi_0(x)\in \C$$ such that
      $$\Psi_0(x_0) = \Om_{x_0}\ \ \ \text{and}\ \ \ \forall x\in I\cap (x_0-\eps_0,x_0+\eps_\Om),\ \ \ \ \ \ F(\Psi_0(x)) = x.$$
    We recall that $\lim\limits_{n\rightarrow +\infty} x_n = x_0$. Thus, there exists $n_0$ such that 
    $$\forall n\ge n_0,\ \ \ \ x_n\in (x_0-\eps_0,x_0+\eps_\Om)\ \ \ \text{and}\ \ \ F(\Psi_0(x_n)) = x_n.$$
    Then, by the definition of $f$ being an infimum, we have 
    $$\forall n\ge n_0,\ \ \ \ f(x_n) \leq J(\Psi_0(x_n)).$$
Passing to the limit, we get
\smaller$$\limsup_{x\rightarrow x_0} f(x) = \lim_{n\rightarrow +\infty} f(x_n) \leq \lim_{n\rightarrow +\infty} J(\Psi_0(x_n))= J\big(\Psi_0(\lim_{n\rightarrow +\infty} x_n)\big) = J(\Psi_0(x_0)) = J(\Om_{x_0}) = f(x_0).$$
\end{itemize}

We then conclude that $$\lim\limits_{x\rightarrow x_0} f(x) = f(x_0).$$

    \textbf{\textit{Monotonicity:}}\

    We now prove that $f$ does not admit a local minimum in the interior of the interval $I$. Let us assume by contradiction that it is not the case: then by the continuity of $f$ there exists a local minimum of $f$ at a point $x^*$ in the interior of $I$. Thus, there exists $\alpha>0$ and $\Om^*\in \C$ such that 
    $$F(\Om^*) = x^*\ \ \ \text{and}\ \ \ \forall x\in  (x^*-\alpha,x^*+\alpha),\ \ \ J(\Om^*) = f(x^*) \leq f(x),$$
    which implies
    $$\forall \Om\in\C\ \text{such that}\ F(\Om)\in (x^*-\alpha,x^*+\alpha),\ \ \ J(\Om^*)\leq J(\Om).$$
    Because of the continuity of $F$ in $(\C,\delta)$, this would imply that $\Om^*$ is a local minimizer of $J$ in $\C$ with respect to $\delta$ which contradicts Hypothesis \ref{hyp4}.  We then conclude that the continuous function $f$ is strictly decreasing on the interval $I$. 
    \end{proof}

    \begin{proposition}
    Let $x\in I$. The problems 
  \begin{enumerate}[label=(\Roman*)]
    \item\label{item:1} $\min\{J(\Om)\ |\ \Om\in \C\ \text{and}\ F(\Om) = x\}$,
    \item\label{item:2} $\min\{J(\Om)\ |\ \Om\in \C\ \text{and}\ F(\Om) \leq x\}$,
    \item\label{item:3} $\min\{F(\Om)\ |\ \Om\in \C\ \text{and}\ J(\Om) = f(x)\}$,
    \item\label{item:4} $\min\{F(\Om)\ |\ \Om\in \C\ \text{and}\ J(\Om) \leq f(x)\}$
\end{enumerate}
are equivalent. In the sense that any solution to one of the problem also solves the other ones.
\end{proposition}
\begin{proof}
Let us prove the equivalence between the four problems. 
\begin{itemize}
\item 
We first show that any solution of \ref{item:1} solves \ref{item:2}: let $\Om_{x}$ be a solution to \ref{item:1}. Then for every set $\Om\in \C$ such that $F(\Om)\leq x$, one has
$$J(\Om) \ge  f\big(F(\Om)\big) \ge f(x) = J(\Om_x),$$
where we used the monotonicity of $f$ given by Proposition \ref{prop:continuity}: therefore $\Om_{x}$ solves \ref{item:2}.
\item Reciprocally, let now $\Om^x$ be a solution of \ref{item:2}: we want to show that $\Om^x$ must satisfy $F(\Om^x)=x$.
We notice that
$$f(x) \ge J(\Om^x)\ge f\big(F(\Om^x)\big) \ge f(x),$$
where the first inequality follows as Problem \ref{item:2} allows more candidates than in the definition of $f$, and the last inequality uses again the monotonicity of $f$.
Therefore $f(x) = f\big(F(\Om^x)\big)$, and since $f$ is continuous and strictly decreasing, we obtain $F(\Om^x) = x$, which implies that the set $\Om^x$ solves \ref{item:1}. 
\end{itemize}
We have proved the equivalence between problems \ref{item:1} and \ref{item:2};
the equivalence between problems \ref{item:3} and \ref{item:4} is shown by similar manipulations. It remains to prove the equivalence between \ref{item:1} and \ref{item:3}. 
\begin{itemize}
\item Let $\Om_{x}$ be a solution of \ref{item:1}, which means that $\Om_x\in \C$ and $J(\Om_x) = f(x)$. Then, for $\Om \in\C$ such that $J(\Om) = f(x)$, we have $$f(x) = J(\Om)\ge f\big(F(\Om)\big).$$
Thus, since $f$ is decreasing, we get $x = F(\Om_x) \leq F(\Om)$, which means that the set $\Om_x$ solves \ref{item:3}. 
\item Let now $\Om_{x}'$ be a solution of \ref{item:3}. We have 
$$f(x)=F(\Om'_x)\ge f\big(F(\Om'_x)\big).$$
Thus, by the monotonicity of $f$, we get $x\ge F(\Om'_x)$. On the other hand, since $\Om'_x$ solves \ref{item:3} and that there exists $\Om_{x}$ solution to \ref{item:1}, we have $F(\Om'_x) \ge x$, which finally gives $F(\Om'_x)=x$ and shows that $\Om'_x$ solves \ref{item:1}.   
\end{itemize}
\end{proof}

\section{Some remarks on the extremal cases $p=1$ and $p=+\infty$}\label{s:extremal}
This section is devoted to the planar setting and the extremal cases $p=1$ and $p=\infty$. On the one hand, Problem $(\mathcal{P}_1)$ is shown to be equivalent to a reverse isoperimetric problem. On the other hand, the case $p=\infty$ corresponds to the problem of minimizing the Hausdorff distance between two convex sets studied in \cite{Ftouhi_zuazua}. 
\subsection{The case $p=1$}
In this case, we retrieve the following problem 
\begin{equation}\label{prob:p1}
    \min \{\int_0^{2\pi} (h_\Om-h)d\theta\ ;\ h\leq h_\Om,\ h+h''\ge 0,\ \text{and}\ \int_0^{2\pi} h(h+h'')d\theta = c\},
\end{equation}
with $c\in[0,|\Om|]$.\vspace{1mm} 

We recall that the perimeter of convex bodies can be computed via their support functions as follows 
$$P(\Omega) = \int_{0}^{2\pi} h_\Omega d\theta\ \ \ \text{and}\ \ \ P(\omega) = \int_{0}^{2\pi} h d\theta,$$
where $h$ and $h_\Omega$ are the support functions of two convex sets $\omega$ and $\Omega$ respectively. This allows us to write 
$$\int_0^{2\pi} (h_\Om-h)d\theta  = P(\Omega)-P(\omega),$$
which shows that Problem \eqref{prob:p1} is equivalent to the following purely geometric one 
\begin{equation}\label{prob:p11}
    \max\{P(\om)\ |\ \om\ \text{is convex}\subset\Om\ \ \text{and}\ |\om|= c\}.
\end{equation}

Problem \eqref{prob:p11} can be interpreted as a \textit{reverse isoperimetric problem}, where the goal is to maximize the perimeter under volume constraint instead of minimizing it. As one expects, for such problems, an extra geometric assumption (such as the convexity for example) is necessary to ensure the existence of a solution. The literature on reverse isoperimetric problems is abundant. For examples of related works, we refer to \cite{henrot_bianchini,bogosel_reverse,zbMATH02574059} for problems under convexity constraint and to 
\cite{zbMATH07601272,arXivvv,zbMATH00796894} for studies where curvature constraints are considered.    

\begin{proposition}
    The optimal set $\om^*$ touches the boundary of $\Om$ at least in two points and the free parts of its boundary, namely the connected components of $\partial \om^*\cap \Om$, are straight lines. 
\end{proposition}
\begin{proof}
The proof relies mainly on parallel
chord movements. More precisely, if the boundary of the solution $\om^*$ contains a polygonal part, one can consider $A$, $B$, $C$ three consecutive corners so that
ABC forms a triangle, see Figure \ref{fig:chord_mouvement}. One can move $B$ along the line passing through $B$ and being parallel to the line $(AC)$.
In this way, the volume is preserved and the perimeter must increase when moving $B$ away from the perpendicular
bisector of the segment $[A, C]$ (which is possible at least in one direction).

    By Theorem \ref{th:main_gamma_2}, any connected part of the boundary of the optimal set $\om^*\cap \Omega$ is polygonal. If we assume that there exists a connected part of $\partial\om^*\cap \Om$ which is not a line, we can perform a parallel chord movement, as in Figure \ref{fig:chord_mouvement}, and increase the perimeter while preserving the value of the area. Thus, every connected part of $\partial\om^*\cap \Om$ is a line. 

    \begin{figure}[!h]
    \centering
    \includegraphics[scale=.8]{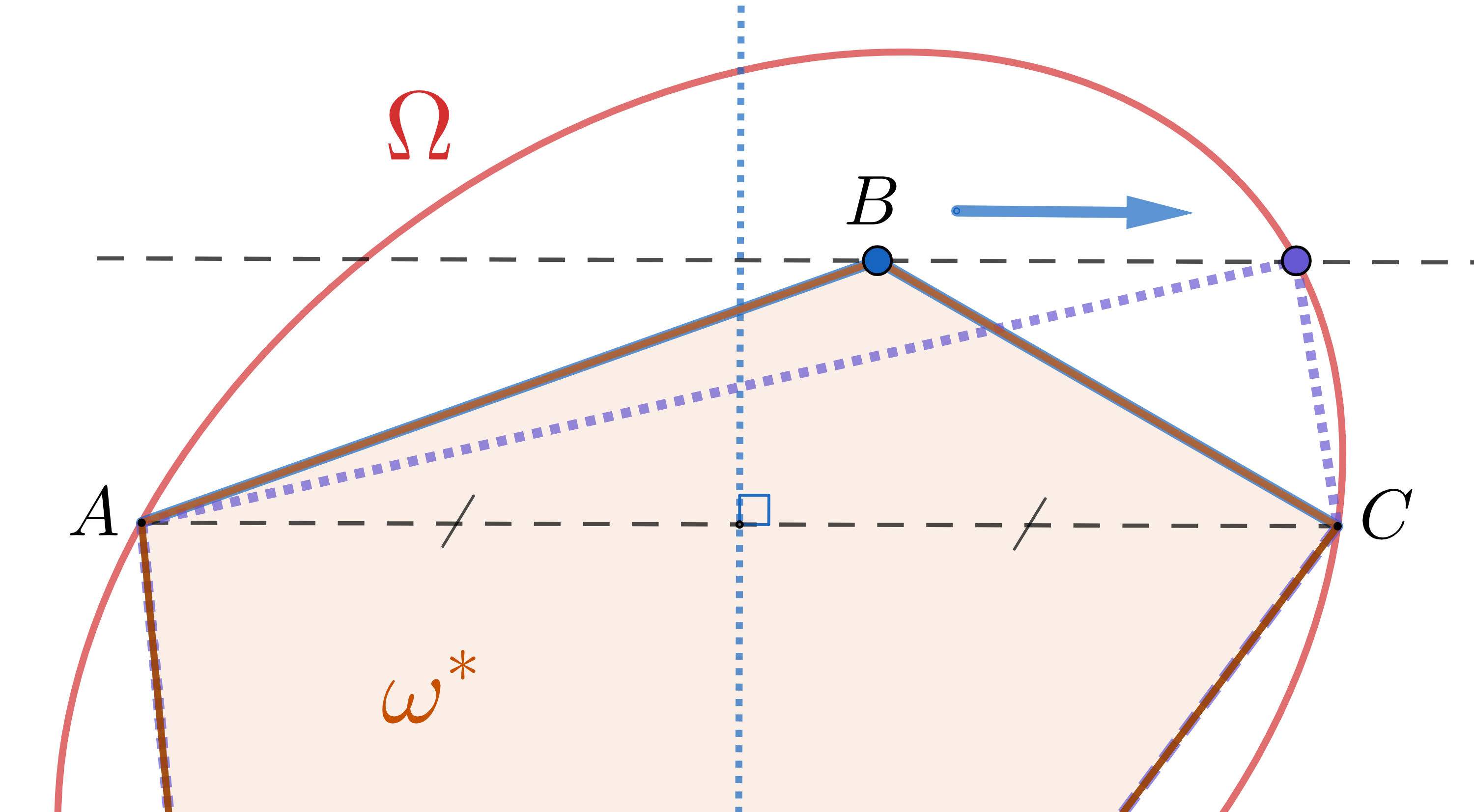}
    \caption{Parallel chord movements increase the perimeter while preserving the area.}
    \label{fig:chord_mouvement}
\end{figure}

At last, since the free part of the boundary of the optimal set $\om^*$ is a line, then $\partial \Om\cap \om^*$ contains at least two different points. 
\end{proof}

\begin{remark}
    It is straightforward that if the container $\Omega$ is a polygon then the optimal solution $\omega^*$ is also polygonal. 
\end{remark}

  We note that Problem \eqref{prob:p11} has recently been completely solved by B. Bogosel when the container $\Omega$ is a ball \cite{bogosel_reverse}. In this case,  the solution $\omega$ is always a polygon having all but one sides equal. The proof is purely geometric and heavily relies on the specific structure of the spherical container $\Omega$. Therefore, as explained in \cite[Remark 11]{bogosel_reverse}, a generalization using similar techniques seems to be out of reach. In the sequel, we state the following conjecture when $\Omega$ is a triangle:
  \begin{conjecture}
      Let $\Omega$ be a triangle of vertices $A$, $B$ and $C$, whose diameter is given by the side $[AB]$ and such that $\widehat{CAB}\ge \widehat{CBA}$ and let $c\in[0,|\Om|]$. The solution of the problem 
      \begin{equation*}
    \max\{P(\om)\ |\ \om\ \text{is convex}\subset\Om\ \ \text{and}\ |\om|= c\}
\end{equation*}
is given by the triangle $MAB$ whose area is equal to $c|\Om|$ such that $M\in[AC]$, see Figure \ref{fig:conjecture_triangle}. 
  \end{conjecture}

    \begin{figure}[!h]
    \centering
    \includegraphics[scale=.6]{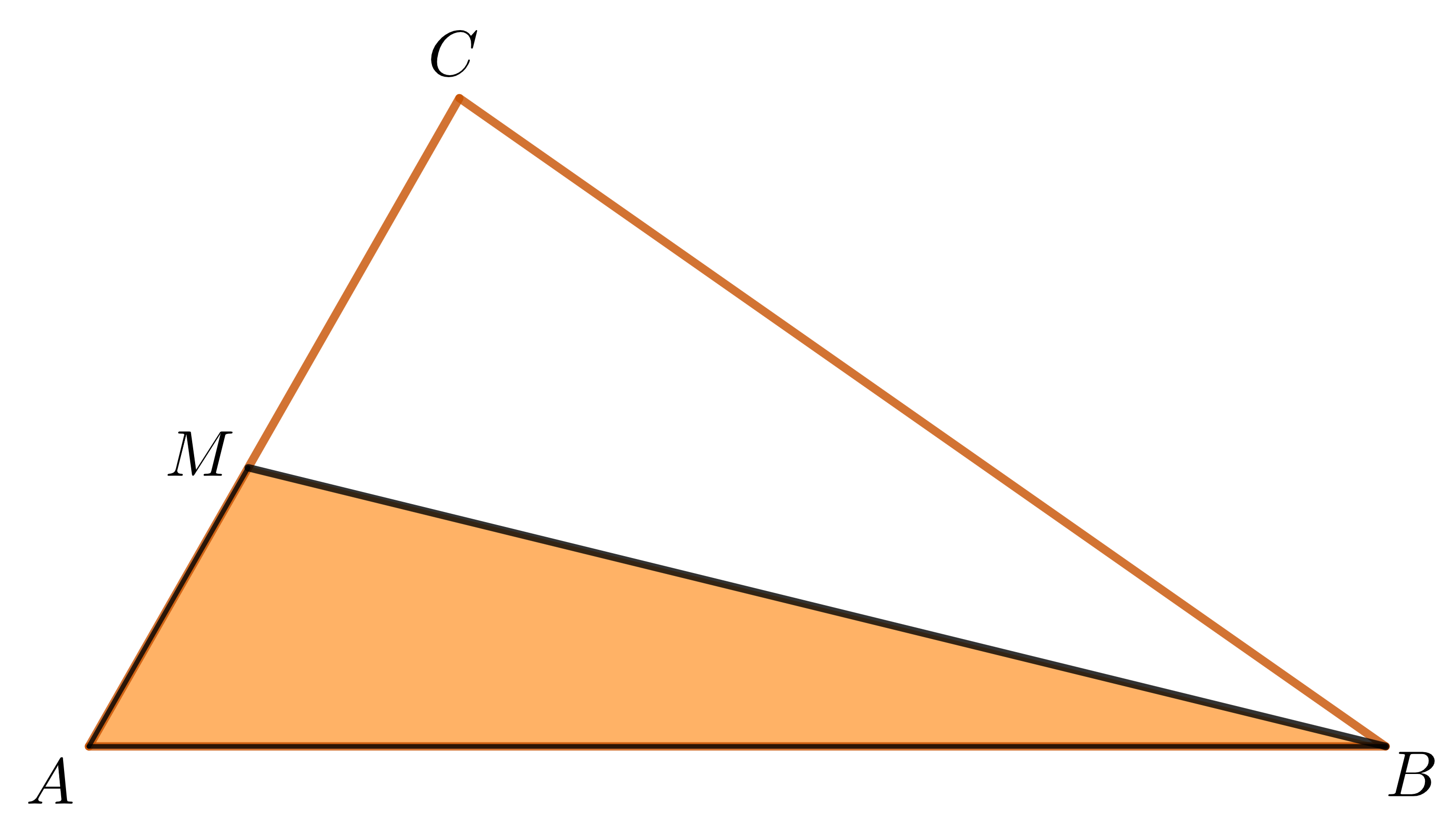}
    \caption{The triangle $MAB$ is conjectured to be the solution of the reverse isoperimetric problem when the contained $\Omega$ is the triangle $ABC$.}
    \label{fig:conjecture_triangle}
\end{figure}

\subsection{The case $p=+\infty$}
The problem  $(\p_\infty)$ can be written as 
$$\min\{d^H(\omega,\Om)\ |\ \omega\subset \Om\ \text{is convex and}\ |\om|=c \},$$
which is, as stated in \cite[Theorem 1]{Ftouhi_zuazua}, equivalent to the problem 
\begin{equation}\label{eq:prob_d}
\min\{|\om|\ |\ \text{$\om$ is convex and}\  h_\Om-d_c\leq h_\om\leq h_\Om\},
\end{equation}
where $d_c$ is a constant depending on $c$.

We are able to characterize the exact solution of the problem in the following case: 
\begin{proposition}
Assume that there exists $d_0>0$ such that $h_\Om''+h_\Om\ge d_0$. For every $d\in [0,d_0]$, the solution of \eqref{eq:prob_d} is given by the inner parallel set $$\Om_{-d}:=\{x\in \Om\ |\ d(x,\partial \Om)\ge d\}.$$
\end{proposition}
\begin{proof}
Let $d\in[0,d_0]$. The function $\psi_d = h_\om-d$ corresponds to the support function of a convex set. Indeed, 
$$\psi_d''+\psi_d = (h_\Om-d)''+(h_\Om-d) = h_\Om''+h_\Om-d\ge h_\Om''+h_\Om -d_0\ge 0. $$
In fact, in this case, the function $\psi_d$ corresponds to the support function of the inner parallel set $\Om_{-d}$. Thus, any set $\om$ satisfying the constraints of problem \eqref{eq:prob_d} satisfies $\psi_d\leq h_\om$, which is equivalent to the inclusion $\Om_{-d}\subset \om$ which yields $|\Om_{-d}|\leq |\om|$. This proves that $\Om_{-d}$ is the only solution of Problem \eqref{eq:prob_d}. 
\end{proof}

One can then state the following corollary:
\begin{corollary}
If the container $\Omega$ is given by $K+ d_0 B_1:=\{x+y\ |\ x\in K\ \text{and}\ y\in B_1\}$, where $K$ is a convex, $d_0>0$ and $B_1$ is the unit ball, then, for every $d\in [0,d_0]$, the solution of \eqref{eq:prob_d} is given by the inner parallel set $\Omega_{-d} = K+ (d_0-d) B_1$.  In particular, if $\Omega$ is a stadium of inradius $r$, then for every $d\in [0,r]$, the solution of \eqref{eq:prob_d} is given by the stadium $\Omega_{-d}$.  
\end{corollary}



\begin{remark}
    We note that the inner parallel sets ($\Omega_{-d}$) do not always provide solutions of the problem $(\p_\infty)$. Indeed, in \cite[Theorem 2]{Ftouhi_zuazua}, the authors show that for a sufficiently small measure, the optimal sensor of the square is given by a given rectangle. 
\end{remark}



\section{Numerical simulations}
\label{s:numerics}

In this section we limit ourselves to the planer setting and present the numerical schemes adopted to solve the problems under consideration. To obtain satisfactory results, we combine two  discretization frameworks: the first, more classical, is based on optimizing the coefficients of the Fourier decomposition of the support function, the second, more recent, was introduced by B. Bogosel \cite{beni_num}, where a rigorous discrete convexity condition is found and that can capture well the presence of segments in the boundary, which is very important in our case since the optimal shapes are proved to contain segments in their boundaries (c.f. Theorem \ref{th:main_gamma_2}). \vspace{2mm}

We recall that we are mainly interested in solving the following type of problems 
$$\min_{\om}\left\{\int_0^{2\pi} |h_\Om-h_\om|^p d\theta\ |\ \omega\subset \Om\ \text{is convex and}\ |\om|=\alpha|\Omega| \right\},$$
where $h_\Om$ and $h_\om$ are the support functions of $\Om$ and $\om$ respectively, $p\in [1,+\infty)\cup \{+\infty\}$ and $\alpha\in [0,1]$. 

By the results recalled in Section \ref{ss:support}, the problem can be formulated in a purely analytic setting as follows 
$$\min_{h\in H^1_{\text{per}}} \left\{\int_0^{2\pi} (h_\Om-h)^p d\theta\ |\ h\leq h_\Om,\ h''+h\ge 0\ \text{and}\   \int_0^{2\pi} (h^2-{h'}^2)d\theta = \alpha|\Omega|\right \},$$
where $H^1_{\text{per}}(0,2\pi)$ is the set of $H^1$ functions that are $2\pi$-periodic.

\subsection{Method 1: Optimizing the Fourier coefficients}\label{ss:fourier_numerics}

In Section \ref{ss:support} we recall that if $\om$ is convex, we have the following formula for the area 
$$|\om| = \frac{1}{2}\int_0^{2\pi} h(h''+h)d\theta=\frac{1}{2} \int_0^{2\pi} (h^2-{h'}^2)d\theta,$$
where $h$ corresponds to the support function of $\om$. 

On the other hand, the inclusion constraint $\om\subset \Om$ can be expressed by $h\leq h_\Om$ on $[0,2\pi]$ and the convexity of the sensor $\om$ is expressed  as
$$h''+h\ge 0,$$
in the sense of distributions. We refer to \cite{schneider} for more details and results on convexity. 

Therefore, we are mainly interested in the analytical problem 
\begin{equation}\label{prob:fonctions}
\left\{\begin{matrix}
\inf\limits_{h\in H^1_{\text{per}}(0,2\pi)} \bigintssss_{0}^{2\pi}(h_\Om-h)^p d\theta,\vspace{2mm}
\\ 
h\leq h_\Om,\vspace{2mm}
\\ 
h''+h\ge 0,\vspace{2mm}
\\
\frac{1}{2}\int_0^{2\pi} h(h''+h)d\theta = \alpha|\Omega|,
\end{matrix}\right.    
\end{equation}

To perform the numerical approximation of optimal shape, we have to retrieve a finite dimensional setting. We then follow the same ideas in \cite{beni_1} and consider Fourier decompositions of the support functions truncated at a certain order $N\ge 1$. We then look for solutions in the set 
$$\mathcal{H}_N:= \left\{\theta\longmapsto a_0 + \sum_{k=1}^N \big(a_k \cos{(k\theta)}+ b_k\sin{(k\theta)}\big) \ \big|\ a_0,\dots,a_N,b_1,\dots,b_N\in \R\right\}.$$
This approach is justified by the following approximation proposition: 
\begin{proposition}(\cite[Section 3.4]{schneider})\\
Let $\Om\in \K^2$ and $\eps>0$. Then there exists $N_\eps$ and $\Om_\eps$ with support function $h_{\Om_\eps}\in \mathcal{H}_{N_\eps}$ such that $d^H(\Om,\Om_\eps)<\eps$. 
\end{proposition}
\begin{remark}
 We refer to \cite{beni_1} for some theoretical convergence results and applications to other different problems.    
\end{remark}
Let us now consider the regular subdivision $(\theta_k)_{k\in \llbracket 1,M \rrbracket }$ of $[0,2\pi]$, where $\theta_k= {2k\pi}/{M}$ and $M\ge 3$. The inclusion constraints $h_\Om-d\leq h\leq h_\Om$ and the convexity constraint $h''+h\ge 0$ are approximated by the following $2M$ linear constraints on the Fourier coefficients: 
$$\forall k\in \llbracket1,M\rrbracket,\ \ \ \ \ \ \left\{\begin{matrix}
a_0+\sum\limits_{j=1}^N \big(a_j \cos{(j\theta_k)}+ b_j\sin{(j\theta_k)}\big)\leq  h_\Om(\theta_k) ,\vspace{2mm}
\\ 
 a_0+\sum\limits_{j=1}^N \big((1-j^2)\cos{(j\theta_k)}a_j + (1-j^2)\sin{(j\theta_k)}b_j \big)\ge 0.
\end{matrix}\right.  $$

As for the area of the set $\omega$, it is approximated by the following quadratic formula:
$$|\om|=\pi a_0^2 + \frac{\pi}{2} \sum\limits_{j=1}^N(1-j^2)(a_j^2+b_j^2).$$

Therefore, the infinitely dimensional problem \eqref{prob:fonctions} is approximated by the following finitely dimensional one
\begin{equation}\label{prob:approx_1}
\left\{\begin{matrix}
\min\limits_{(a_0,a_1,\dots,a_N,b_1,\dots,b_N)\in \R^{2N+1}} \bigintsss_0^{2\pi} \left(h_\Om(\theta) - a_0 - \sum\limits_{j=1}^N \big(a_j \cos{(j\theta)}+b_j\sin{(j\theta)}\big) \right)^p d\theta,\vspace{2mm}
\\ 
\forall k\in \llbracket1,M\rrbracket,\ \ \  a_0+\sum\limits_{j=1}^N \big(a_j \cos{(j\theta_k)}+ b_j\sin{(j\theta_k)}\big)\leq  h_\Om(\theta_k),\vspace{2mm}
\\ 
\forall k\in \llbracket1,M\rrbracket,\ \ \   a_0+\sum\limits_{j=1}^N \big((1-j^2)\cos{(j\theta_k)}a_j + (1-j^2)\sin{(j\theta_k)}b_j \big)\ge 0,\vspace{2mm}
\\
\pi a_0^2 + \frac{\pi}{2} \sum\limits_{j=1}^N(1-j^2)(a_j^2+b_j^2) = \alpha|\Omega|. 
\end{matrix}\right.    
\end{equation}

\subsection{Method 2: A new method based on a rigorous convexity parametrization}\label{ss:beni_numerics}
The method based on the optimization of the Fourier coefficients of the support function presented in Section \ref{ss:fourier_numerics} showed to be less efficient when the optimal shape is not strictly convex, which is the case for the problems considered in this paper as shown in Theorem \ref{th:main_gamma_2}. We then propose to combine it with the recent parametrization introduced by B. Bogosel in \cite{beni_num} that allows to overcome this difficulty and capture segments on the boundaries of the optimal shapes.

It is classical that given the support function $h$ of a strictly convex shape $\Om\subset\R^2$, $h$ is of class $C^1$ and a parametric representation of $\partial \Om$ is given by  
$$\left\{\begin{matrix}
x(\theta) = h(\theta)\cos{\theta} - h'(\theta)\sin{\theta},\vspace{2mm}
\\ 
y(\theta) = h(\theta)\sin{\theta} + h'(\theta)\cos{\theta}.
\end{matrix}\right.$$


For $N\ge 3$, we consider an equidistant partition of $[0,2\pi]$ given by $\theta_j= 2\pi j/N$ and denote $h_j := h(\theta_j)$, where $j\in \llbracket 0,N \rrbracket$. We then use the following approximations: 

\begin{itemize}
\item  The objective function is approximated as follows 
$$\int_0^{2\pi} (h_\Om-h)^p d\theta \approx  \frac{2\pi}{N} \sum_{j=1}^n (h_\Om(\theta_j) - h_j)^p.$$

\item The main novel idea of \cite{beni_num} is to approximate the radius of curvature $h''+h$ as follows 
$$h''(\theta_j)+h(\theta_j) = h_j + \frac{h_{j+1}-2h_j +h_{j-1}}{2-2 \cos{\frac{2\pi}{N}}},\ \ \ j\in \llbracket 1,N \rrbracket,$$
which provides the following rigorous discrete convexity condition
$$h_{j+1} + h_{j-1}-2 h_j\cos{\frac{2\pi}{N}}\ge 0,\ \ \ j\in \llbracket 1,N \rrbracket.$$

 \item The inclusion $\om\subset\Om$ (equivalent to $h\leq h_\Om$) is then approximated via the inequalities 
$$h_j\leq h_\Om(\theta_j),\ \ \ j\in \llbracket 1,N \rrbracket.$$

\item As for the area of $\om$, we write
\begin{align*}
|\om| &=  \frac{1}{2}\int_{0}^{2\pi} h(h''+h)d\theta \\
&\approx  \frac{\pi}{N}\sum_{j=1}^N h_j\left(h_j + \frac{h_{j+1}-2h_j +h_{j-1}}{2-2 \cos{\frac{2\pi}{N}}}\right) \\
&= \frac{\pi/N}{2-2\cos{\frac{2\pi}{N}}}\sum_{j=1}^N h_j(h_{j+1} + h_{j-1}-2 h_j\cos{\frac{2\pi}{N}}).
\end{align*}

\end{itemize}

Problem \eqref{prob:fonctions} is then approximated as follows:
\begin{equation}\label{prob:approx_2}
\left\{\begin{matrix}
\min\limits_{(h_1,\dots,h_N)\in \R^{N}} 
 \frac{2\pi}{N} \sum\limits_{j=1}^N (h_\Om(\theta_j) - h_j)^p,\vspace{3mm}
\\ 
\forall j\in \llbracket1,N\rrbracket,\ \ \  h_j \leq  h_\Om(\theta_j),\vspace{3mm}
\\ 
\forall j\in \llbracket1,N\rrbracket,\ \ \   h_{j+1} + h_{j-1}-2 h_j\cos{\frac{2\pi}{N}}\ge 0,\vspace{3mm}
\\
\sum\limits_{j=1}^N \left(h_j h_{j+1} + h_{j-1}h_j-2 h_j^2\cos{\frac{2\pi}{N}}\right) = \frac{2\alpha N(1-\cos{\frac{2\pi}{N}})}{\pi}. 
\end{matrix}\right.    
\end{equation}

\begin{remark}
    In practice, it turns out that the use of the optimal shapes obtained with \textbf{Method 1} as initial shapes for the second \textbf{Method 2} allows to obtain quite satisfactory results, see Figures \ref{fig:history} and \ref{fig:history2}.  
\end{remark}


\subsection{Computational experiments}

\subsubsection{Comparison between the two methods}

Since the optimal sets are shown to contain polygonal parts on their boundaries (see Theorem \ref{th:main_gamma_2}), the method based on the rigorous convexity parametrization (Section \ref{ss:beni_numerics}) is likely to provide better results than the classical method via Fourier coefficients (Section \ref{ss:fourier_numerics}). In what follows, we present some numerical simulations supporting this claim. 
\begin{table}[h!]
\begin{tabular}{|c|c|}
\hline Method 1  & Method 2 \\ \hline
\includegraphics[scale=.35]{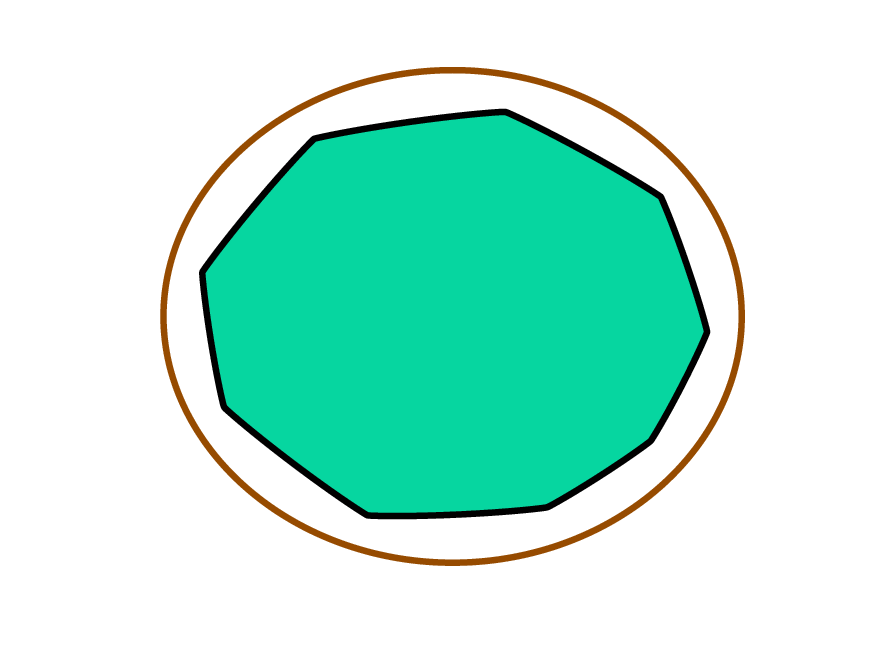} & \includegraphics[scale=.35]{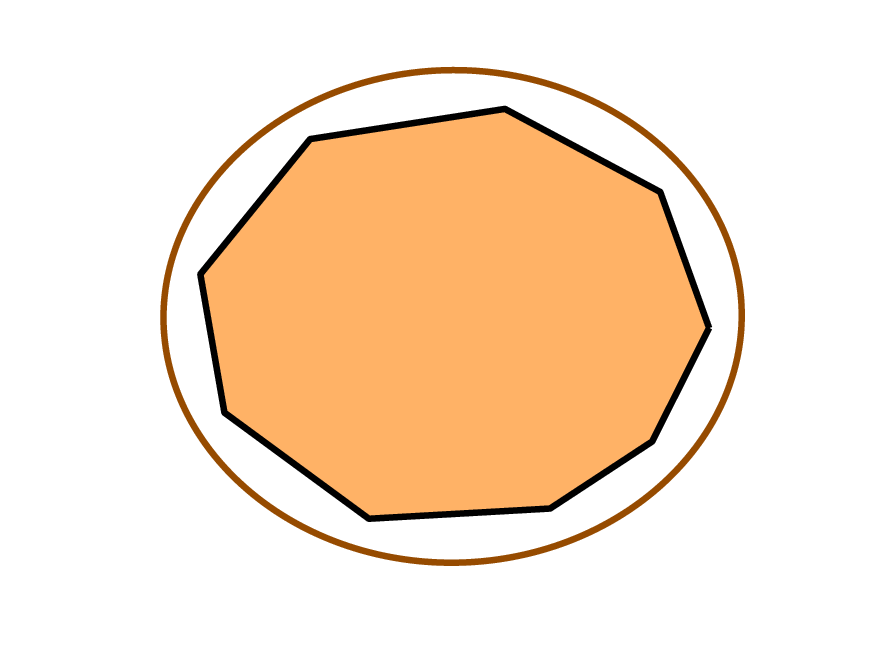} \\ \hline
 Energy = 0.942 & Energy = 0.913 \\ \hline
\end{tabular}
\caption{Obtained optimal sensors with \textbf{Methods 1} and \textbf{2} for $p=10$ and $\alpha = 0.7$.}
\end{table}

\begin{figure}[h!]
    \centering
    \includegraphics[width=0.44\linewidth]{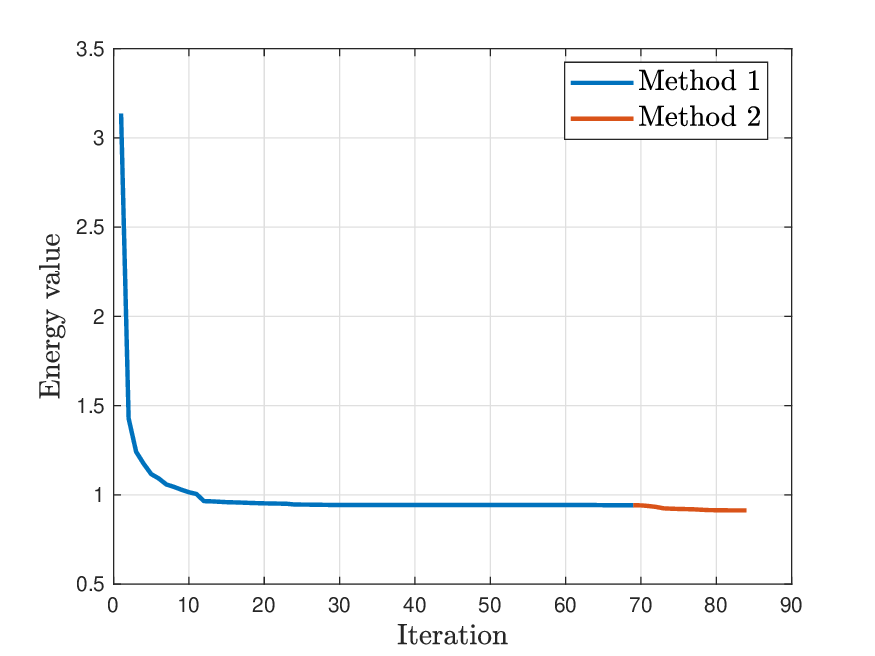}
    \includegraphics[width=0.44\linewidth]{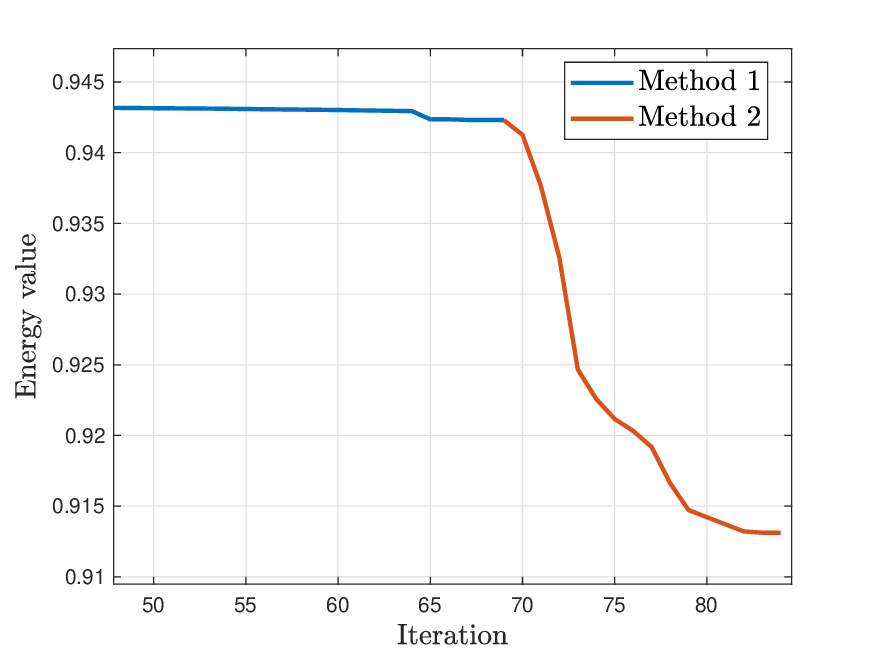}
        \caption{First example for the history of convergence of the two methods (with a zoom in the figure in right).}
    \label{fig:history}
\end{figure}

\begin{table}[h!]
\begin{tabular}{|c|c|}
\hline Method 1  & Method 2 \\ \hline
\includegraphics[scale=.4]{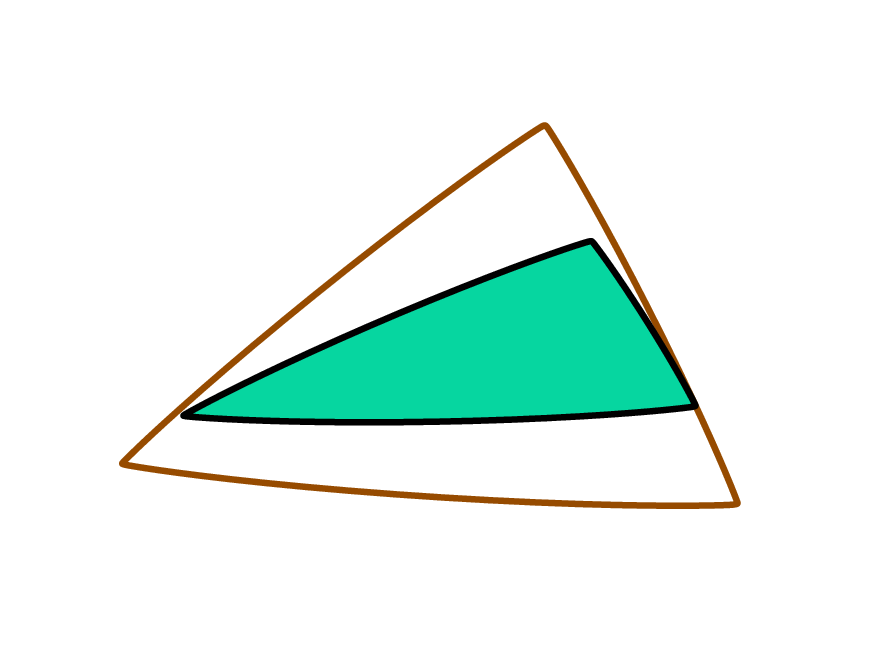} & \includegraphics[scale=.4]{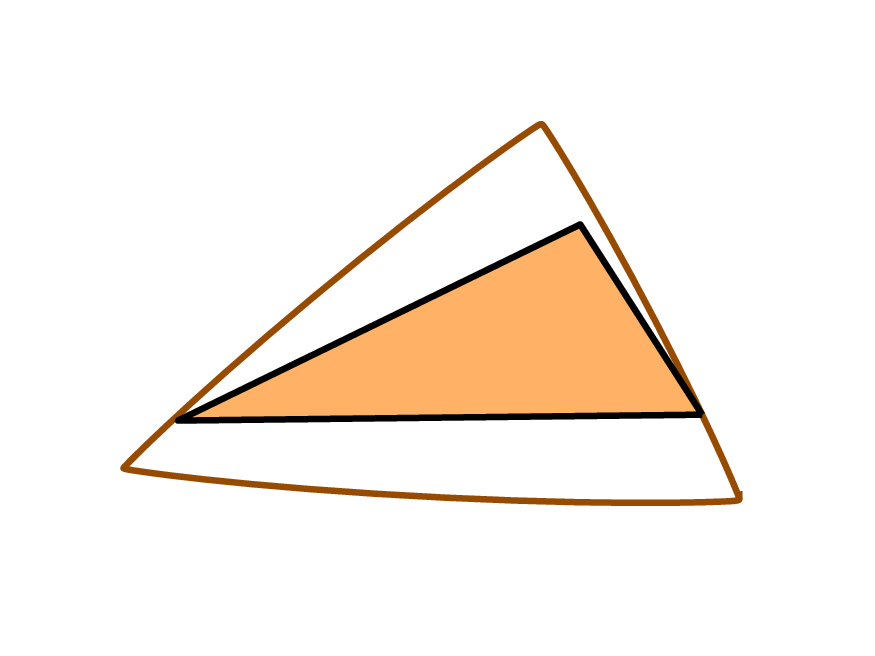} \\ \hline
 Energy = 1.185 & Energy = 1.053 \\ \hline
\end{tabular}
\caption{Obtained optimal shapes with \textbf{Methods 1} and \textbf{2} for $p=4$ and $\alpha = 0.4$.}
\end{table}

\begin{figure}[h!]
    \centering
    \includegraphics[width=0.44\linewidth]{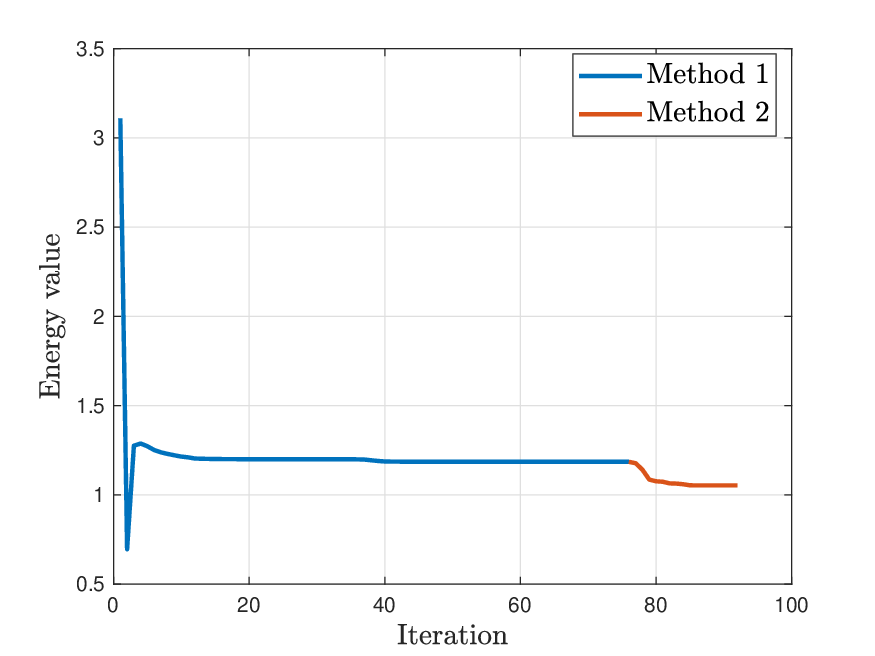}
    \includegraphics[width=0.44\linewidth]{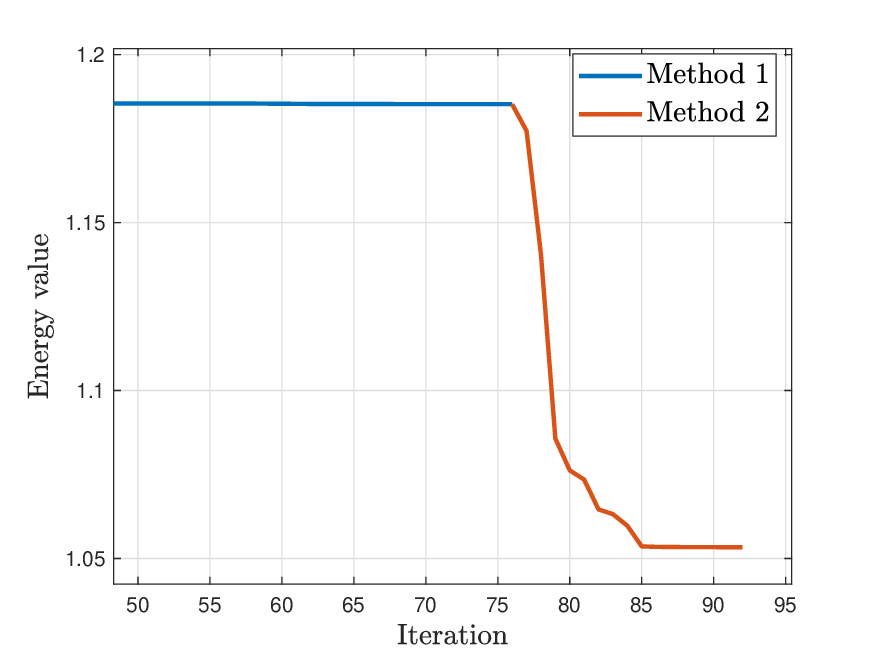}
        \caption{Second example for the history of convergence of the two methods (with a zoom in the figure in right).}
    \label{fig:history2}
\end{figure}


\subsubsection{Obtained optimal shapes}
In this section, we present some obtained numerical results for two chosen containers $\Omega$ and different values of $p$ and of the fraction $\alpha$.

It is worth noting that the problems considered are complex with many local minima. Therefore, multiple runs of the algorithm with different random initializations for the optimization variables are used. The results giving the least values of the energy are shown in the figures below.

\begin{table}[h]
\begin{tabular}{|c|c|c|c|}
\hline
             & $p=1$ & $p=2$ & $p=8$ \\ \hline
$\alpha=0.2$ &   \includegraphics[scale=.3]{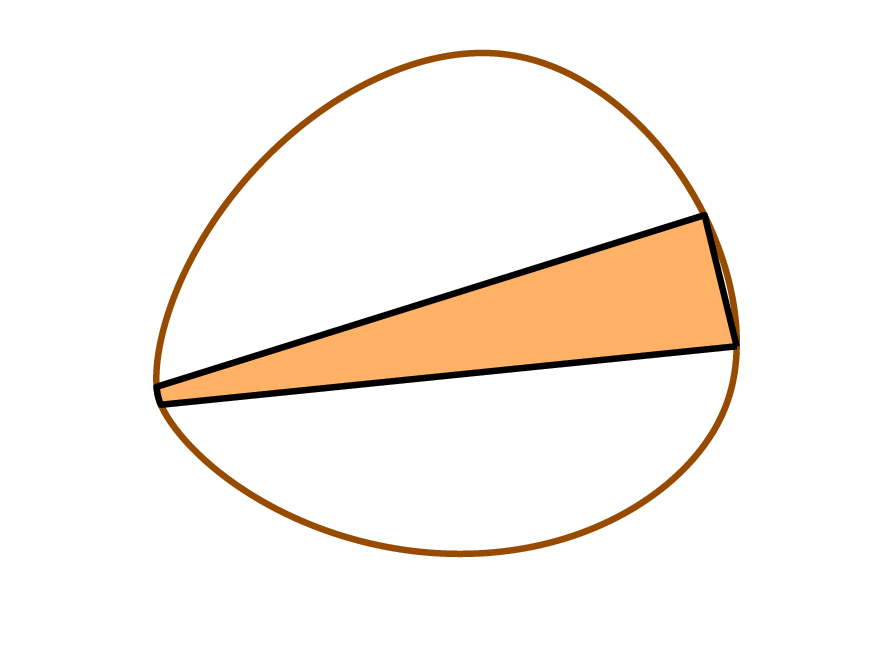}     &  \includegraphics[scale=.3]{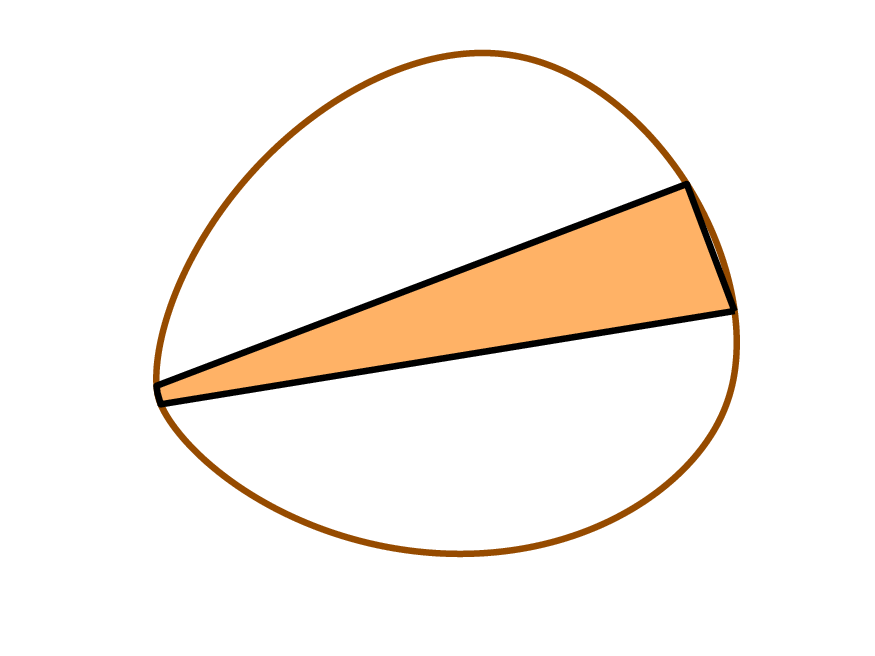}     &  \includegraphics[scale=.3]{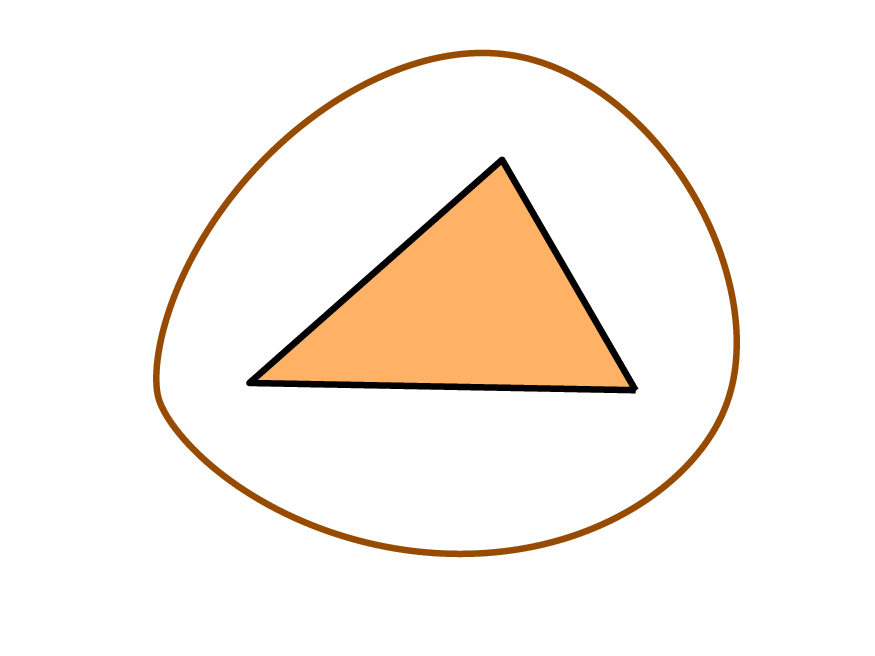}     \\ \hline
$\alpha=0.5$ &  \includegraphics[scale=.3]{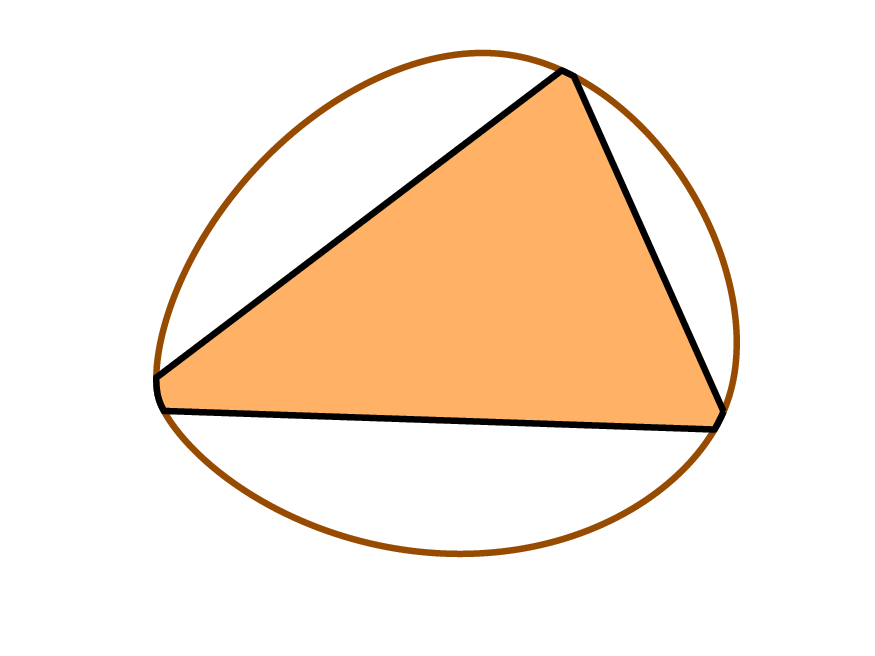}     &    \includegraphics[scale=.3]{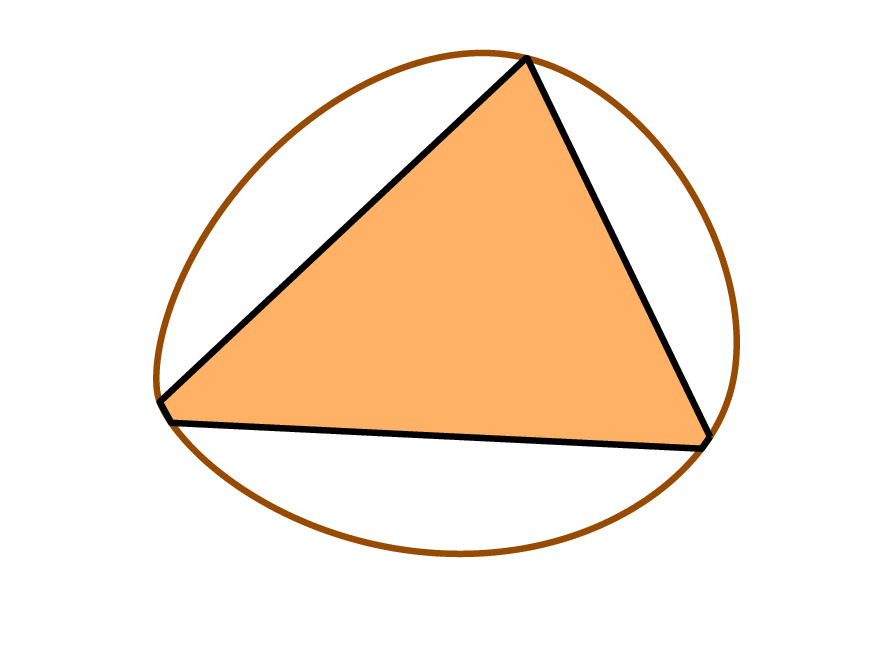}   &   \includegraphics[scale=.3]{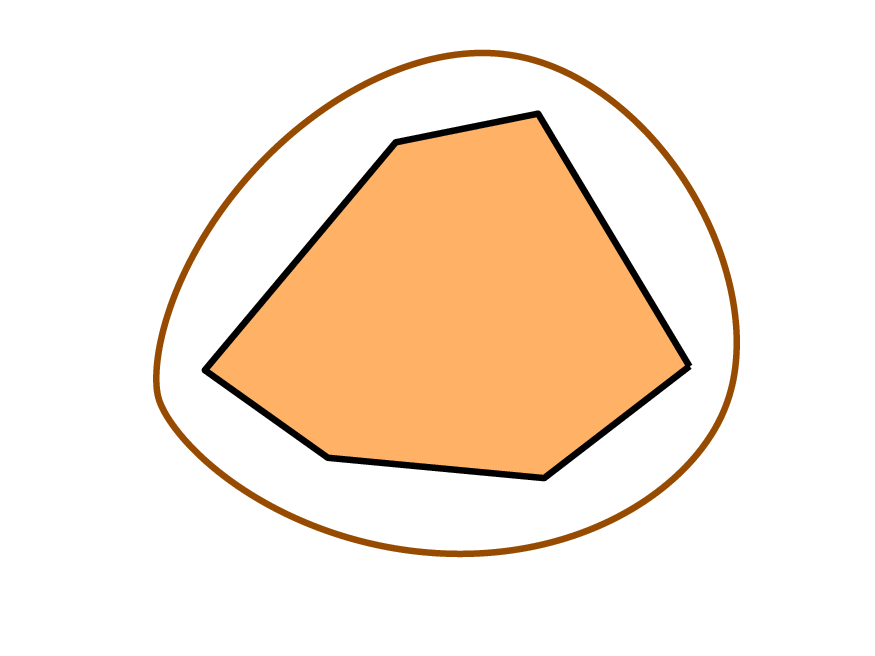}    \\ \hline
$\alpha=0.8$ & \includegraphics[scale=.3]{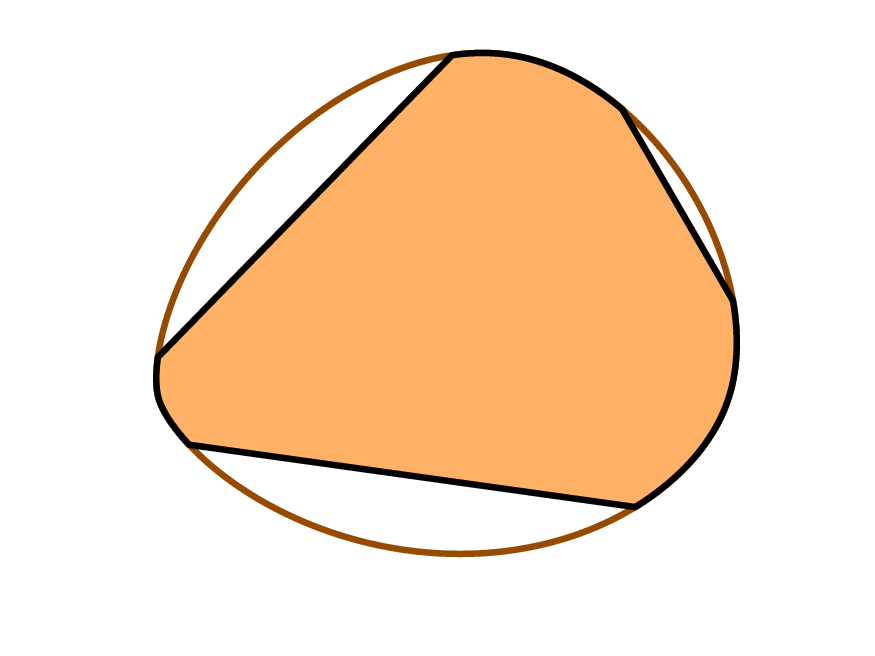}      &  \includegraphics[scale=.3]{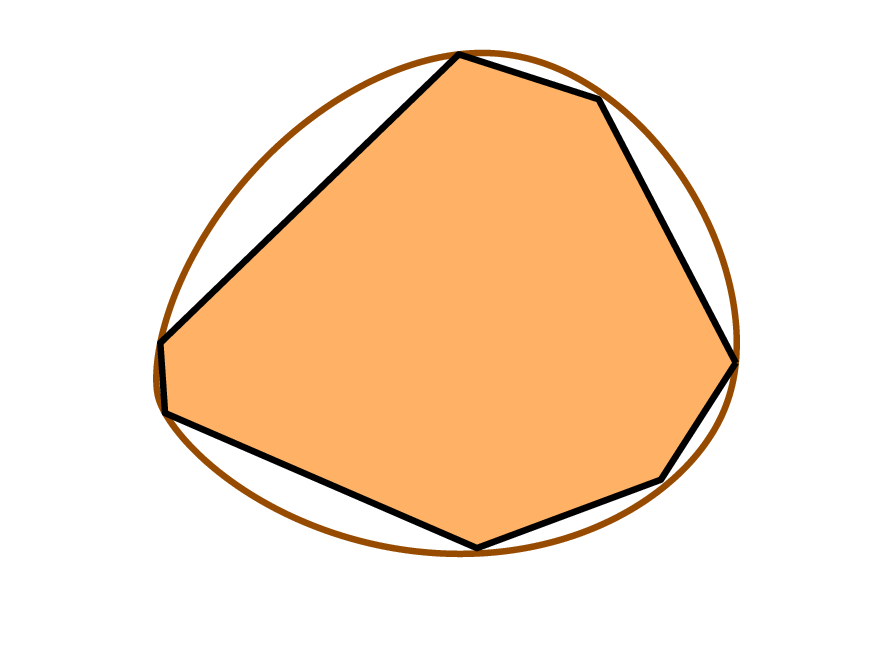}      &  \includegraphics[scale=.3]{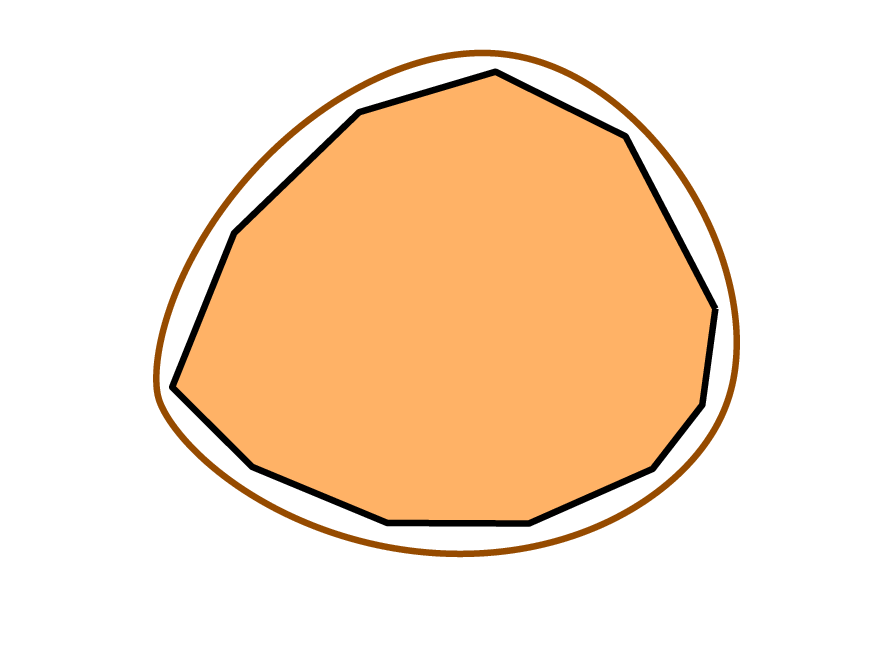}      \\ \hline
\end{tabular}
\caption{Obtained optimal shapes for $p\in\{1,2,8\}$ and $\alpha\in\{0.2,0.5,0.8\}$.}
\end{table}

\begin{table}[]
\begin{tabular}{|c|c|c|c|}
\hline
             & $p=1$ & $p=2$ & $p=8$ \\ \hline
$\alpha=0.2$ &   \includegraphics[scale=.32]{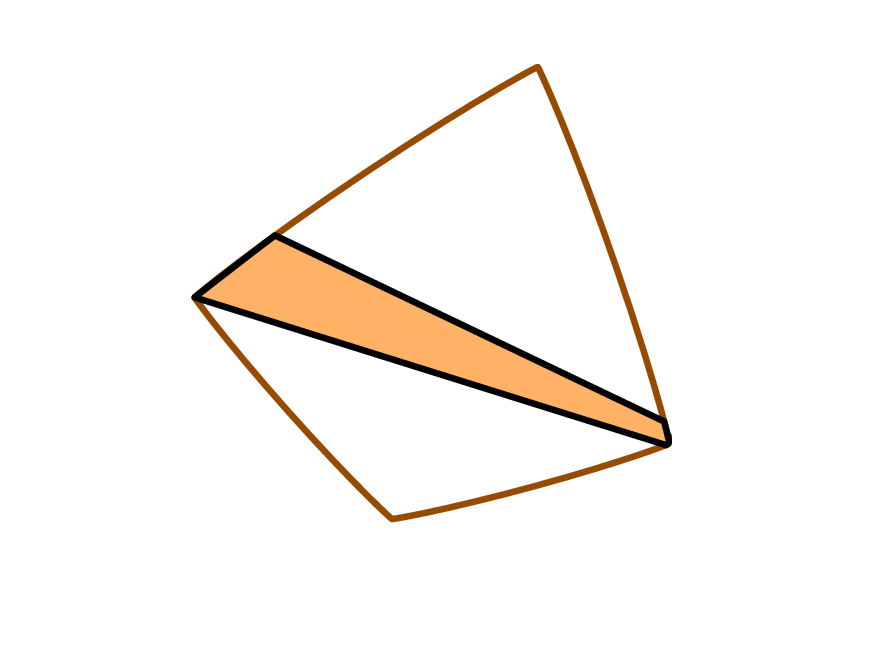}     &  \includegraphics[scale=.32]{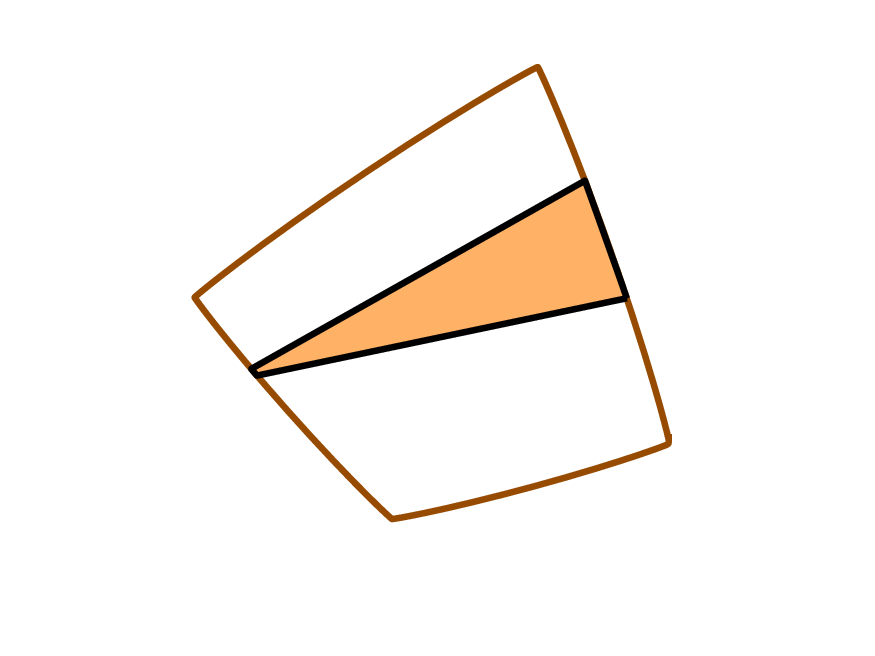}     &  \includegraphics[scale=.32]{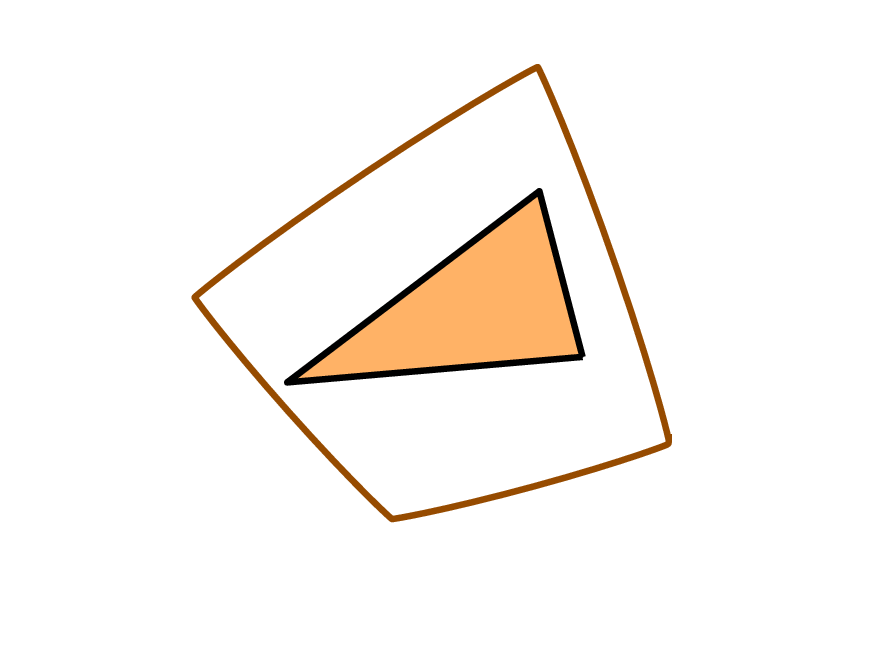}     \\ \hline
$\alpha=0.5$ &  \includegraphics[scale=.32]{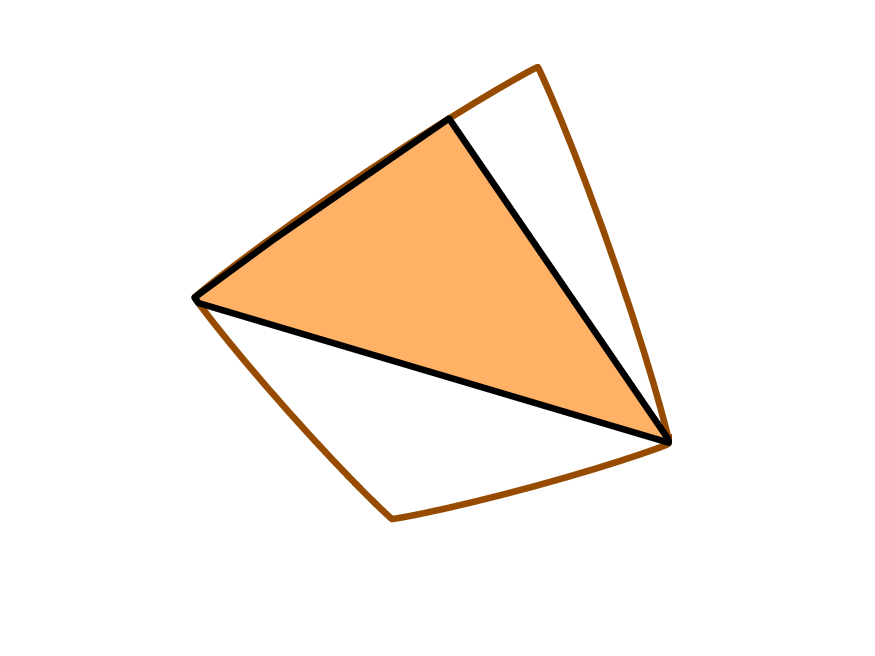}     &    \includegraphics[scale=.32]{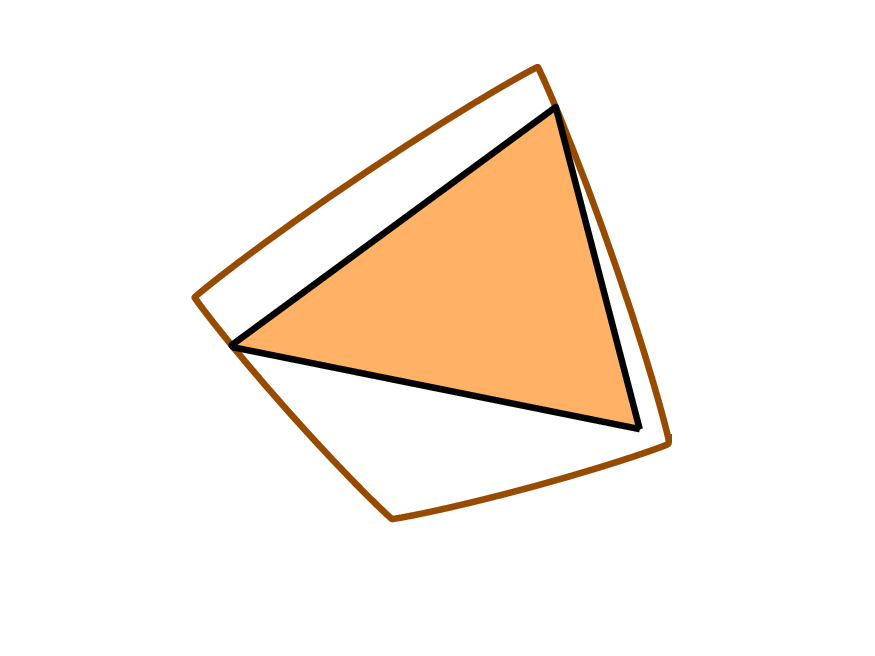}   &   \includegraphics[scale=.32]{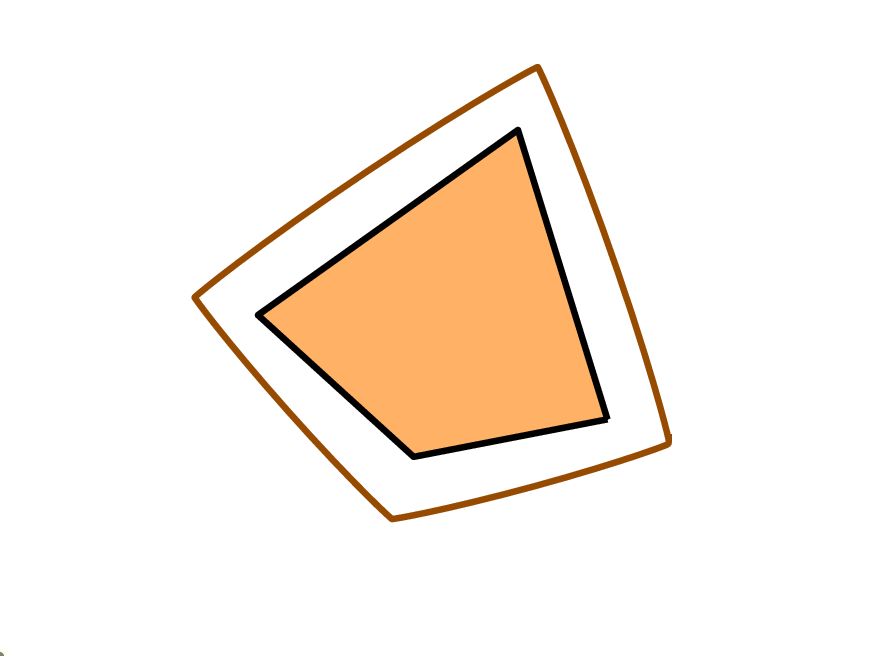}    \\ \hline
$\alpha=0.8$ & \includegraphics[scale=.32]{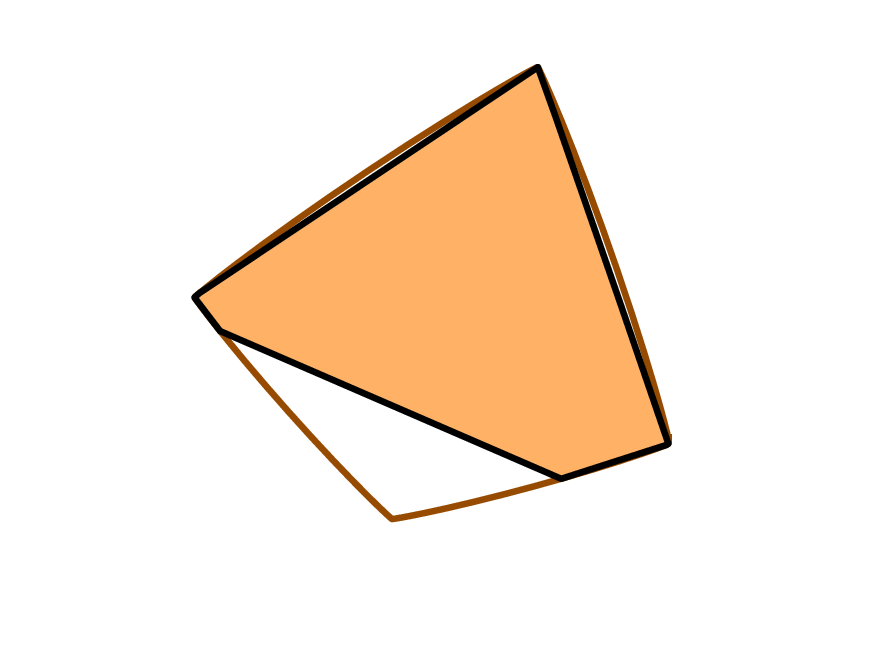}      &  \includegraphics[scale=.32]{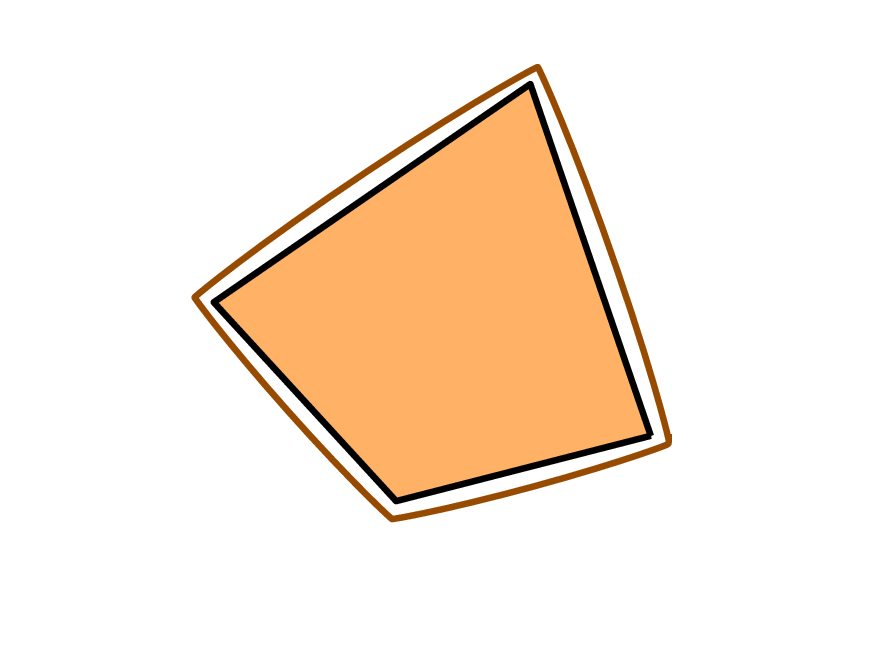}      &  \includegraphics[scale=.32]{example2_a_8_p_8.eps}      \\ \hline
\end{tabular}
\caption{Obtained optimal shapes for $p\in\{1,2,8\}$ and $\alpha\in\{0.2,0.5,0.8\}$.}
\end{table}

    

\section{Conclusions and perspectives}
The present paper addresses the problem of using the classic $L^p$-metric defined on convex bodies as a criterion for approximating a convex set $\Omega$ via convex subsets $\omega$ of a given measure. Here, we outline some key problems and research avenues that we believe merit further investigation.

\begin{itemize}
   \item {\bf More general geometries.} The analysis in this paper is limited to the case of convex domains. The same problems make sense in the absence of convexity restrictions, both on the domain $\Omega$ and the subdomains $\omega$.
   
   \item{\bf Higher dimensions.} The present paper presents a careful analysis of $2-d$ optimal shapes employing exhaustively the support function. The extension of this analysis and computational study to higher dimensions is a challenging topic.
   
     \item {\bf Other geometric constraints.} It would be interesting to explore  the same problems under other geometric constraints on the subsets, such as, for instance,  the perimeter constraint    $$P(\omega) = \int_{\mathcal{S}^{n-1}} h d\mathcal{H}^{n-1},$$
    where $h$ is the support function of $\om$. This leads to the problems
    $$\min\{\int_{\mathcal{S}^{n-1}} (h_\Om-h)^p d\mathcal{H}^{n-1}\ |\ h\leq h_\Om,\ h''+h\ge 0\ \text{and}\ \int_{\mathcal{S}^{n-1}} h d\mathcal{H}^{n-1} = c|\Omega|\},$$
    that are intuitively expected to be easier than those addressed in this paper. However, they  present their own unique challenges that deserves specific attention. \vspace{2mm}
    \item {\bf Other metrics.} It is also interesting to consider the same optimal shape design problems for different metrics. An illustrative example is provided in \cite{lemenant_convex}, where the authors consider the mean distance functional 
    $$J_\Om(K):= \int_{\Om} d(x,K)^p dx,$$
    where $K\subset\Om$ are two convex sets. The authors consider both volume and perimeter constraints on the subsets $K$ and provide explicit formulas for the first and the second order shape derivatives of $J_\Om$. In the same spirit of Modica--Mortola \cite{modica_mortola}, an approximation of the functional $J_\Om$ via relevant PDEs' solutions is introduced and studied. However, as far as we know, the computational analysis and experiments for these problem were not undertaken. \vspace{2mm}
    \item {\bf Varadhan's approximation of the distance function.} It would be valuable to explore alternative approximations of the for the mean distance, beyond those used in \cite{lemenant_convex}. In a similar vein, proving $\Gamma$-convergence results for these new approximations would also be of interest. 
    
   It would  be relevant to further exploit the use a suitable approximations of the distances functions in terms of the solution of simple elliptic PDEs, inspired by the following classical result:
\begin{theorem}\label{th:varadhan}(\cite[Theroem 2.3]{varadhan})

Let $\Omega$ be an open subset of $\R^n$ and $\varepsilon>0$. We consider the problem 
\begin{equation*}\label{prob:varadhan}
\left\{\begin{matrix}
w_\eps-\eps \Delta w_\eps=0 \ \ \ \text{in $\Omega$},
\\ 
\ \ \ \ \ \  \ \ \ \ \ \ w_\eps = 1\ \ \ \text{on $\partial \Omega$}.
\end{matrix}\right.    
\end{equation*}
 This has been successfully developed in \cite{crane}, employing the parabolic counterparts of this result.

$$\lim_{\eps\rightarrow 0} -\sqrt{\eps}\ln{w_\eps(x)} = d(x,\partial \Omega):= \inf_{y\in \partial \Omega} \|x-y\|,$$
uniformly over compact subsets of $\Om$.  
\end{theorem}

    \begin{figure}[!h]
    \centering
    \includegraphics[scale=.5]{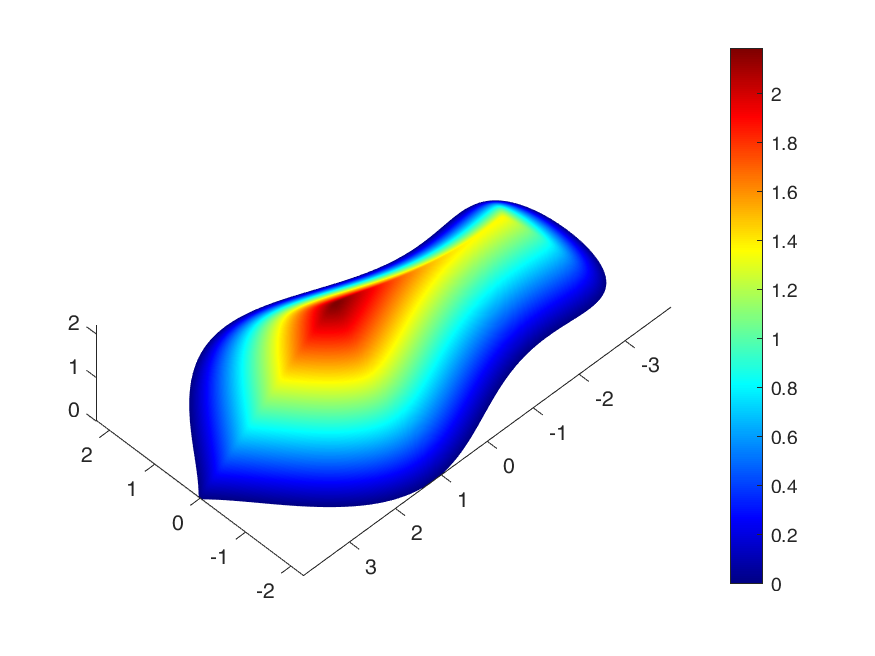}
    \includegraphics[scale=.5]{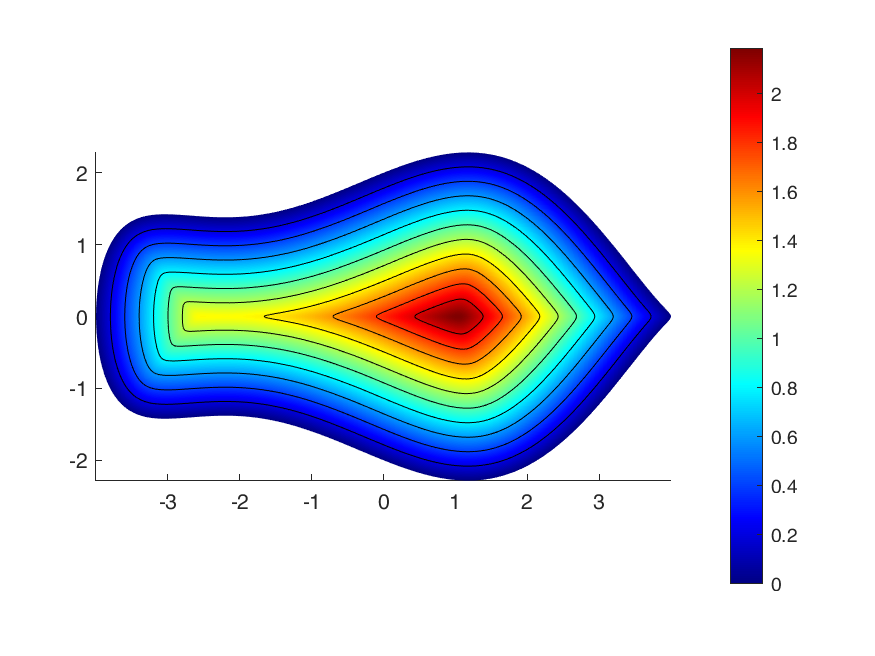}
    \caption{Varadhan's result for the approximation of the distance function.}
    \label{fig:vara}
\end{figure}
    
\end{itemize}

\section*{Acknowledgement}


I. Ftouhi is supported
by the Alexander von Humboldt Foundation via a Postdoctoral fellowship.

E. Zuazua has been supported by the Alexander von Humboldt-Professorship program, the European Union's Horizon
Europe MSCA project ModConFlex, the Transregio 154 Project of the DFG, AFOSR 24IOE027 project, grants
PID2020-112617GB-C22, TED2021-131390B-I00 of MINECO and PID2023-146872OB-I00 of MICIU (Spain), and
the Madrid Government - UAM Agreement for the Excellence of the University Research Staff in the context of the
V PRICIT (Regional Programme of Research and Technological Innovation).

\bibliographystyle{plain}
\bibliography{sample}

\vspace{.3cm}
\smaller\smaller \smaller
(\textbf{Zakaria Fattah})
\textsc{Mathematics and Computer Science Department, ENSAM of Meknes, University of Moulay Ismail, Marjane II, AL Mansour, B.P. 15290, 50050 Meknes, Morocco.}\vspace{1mm}

\textit{Email address:} \textbf{\texttt{z.fattah@edu.umi.ac.ma}}\vspace{4mm}

(\textbf{Ilias Ftouhi})
\textsc{[1] Chair for Dynamics, Control, Machine Learning and Numerics, Alexander von Humboldt-Professorship, Department of Mathematics,  Friedrich-Alexander-Universit\"at Erlangen-N\"urnberg, 91058 Erlangen, Germany,}\vspace{1mm}

\textsc{[2] King Fahd university of Petroleum and Minerals, Department of Mathematics, 31261 Dhahran, Saudi Arabia.}\vspace{1mm}

\textit{Email address:} \textbf{\texttt{ilias.ftouhi@fau.de}}\vspace{4mm}

(\textbf{Enrique Zuazua}) \textsc{[1] Chair for Dynamics, Control, Machine Learning and Numerics, Alexander von Humboldt-Professorship, Department of Mathematics,  Friedrich-Alexander-Universit\"at Erlangen-N\"urnberg,
	91058 Erlangen, Germany,}\vspace{1mm}

\textsc{[2] Chair of Computational Mathematics, Fundaci\'{o}n Deusto,
	Av. de las Universidades, 24,
	48007 Bilbao, Basque Country, Spain,}\vspace{1mm}

 \textsc{[3] Departamento de Matem\'{a}ticas,
	Universidad Aut\'{o}noma de Madrid,
	28049 Madrid, Spain.}\vspace{1mm}

\textit{Email address:} \textbf{\texttt{enrique.zuazua@fau.de}}

\end{document}